\documentclass[11pt]{article}
\usepackage{amsmath}
\usepackage{amsfonts}
\usepackage{color}
\usepackage{latexsym}
\usepackage{tablefootnote}
\usepackage{epsfig}
\usepackage{multirow}
\usepackage{float}
\usepackage{array}
\usepackage{amssymb,amsthm}

\newcommand*{\QEDA}{\hfill\hbox{\vrule width1.0ex height1.0ex}}

\usepackage{makecell,booktabs}
\usepackage[version=4]{mhchem}
\usepackage[strict]{changepage}

\usepackage{amsmath}
\usepackage{amsfonts}
\usepackage{latexsym}
\usepackage{amssymb}

\newtheorem{thm}{Theorem}[section]
\newtheorem{theorem}[thm]{Theorem}

\newtheorem{lemma}[thm]{Lemma}

\newtheorem{corollary}[thm]{Corollary}
\newtheorem{definition}[thm]{Definition}

\newtheorem{remark}[thm]{Remark}

\newcommand{\beq}{\begin{equation}}
\newcommand{\eeq}{\end{equation}}
\newcommand{\beqa}{\begin{eqnarray}}
\newcommand{\eeqa}{\end{eqnarray}}
\newcommand{\beqas}{\begin{eqnarray*}}
\newcommand{\eeqas}{\end{eqnarray*}}
\newcommand{\bi}{\begin{itemize}}
\newcommand{\ei}{\end{itemize}}

\newcommand{\vgap}{\vspace{.1in}}
\newcommand{\nn}{\nonumber}

\newcommand{\R}{\mathbb{R}}

\newcommand{\lam}{{\lambda}}

\newcommand{\inner}[2]{\langle #1,#2\rangle}

\newcommand{\argmin}{\mathrm{argmin}\,}

\newcommand{\dom}{\mathrm{dom}\,}

\newcommand{\bConv}[1]{\overline{\mbox{\rm Conv}}\,(\R^{#1})}

\newcommand{\tx}{\tilde x}
\newcommand{\ty}{\tilde y}

\def\NN{{\mathbb{N}}}

\begin{document}
\title{Average Curvature FISTA for Nonconvex Smooth\\ Composite Optimization Problems}
\date{May 13, 2021 \\ first revision: March 10, 2022 \\
second revision: February 3, 2023}
\maketitle

\begin{center}
			\textsc{Jiaming Liang 
   \footnote{Department of Computer Science, Yale University, New Haven, CT 06511.
			(email: {\tt jiaming.liang@yale.edu}).
			This author
			was partially supported by
			ONR Grant
			N00014-18-1-2077 and NSF grant CCF-1740776 through IDEaS-TRIAD Research Scholarship from Georgia Tech.
			}
            and 
			Renato D.C. Monteiro} \footnote{School of Industrial and Systems
			Engineering, Georgia Institute of
			Technology, Atlanta, GA, 30332-0205.
			(email: {\tt monteiro@isye.gatech.edu}). This author
			was partially supported by ONR Grant
			N00014-18-1-2077 and	AFOSR Grant FA9550-22-1-0088.}
\end{center}

\begin{abstract}


A previous authors' paper introduces an
accelerated composite gradient (ACG) variant, namely
AC-ACG, for solving nonconvex smooth composite optimization (N-SCO) problems.
In contrast to other ACG variants, AC-ACG
estimates the local upper curvature of the N-SCO problem
by using the average of the observed
upper-Lipschitz curvatures obtained during the previous iterations,
and uses this estimation and two composite resolvent evaluations to
compute the next iterate. This paper presents an alternative
FISTA-type ACG variant, namely AC-FISTA,
which has the following additional
features:
i) it performs an average of one composite resolvent evaluation per iteration;
and ii) it estimates the local upper curvature  by using the
average of the previously observed
upper (instead of upper-Lipschitz) curvatures.
These two properties acting together yield  a practical
AC-FISTA variant which
substantially outperforms earlier ACG variants, including
the  AC-ACG variants discussed in the aforementioned authors' paper.


\end{abstract}

\vspace{.1in}


\vspace{.1in}


\section{Introduction}
This paper studies a FISTA-type accelerated composite gradient (ACG) algorithm, namely the AC-FISTA method, for solving the nonconvex smooth composite optimization (N-SCO) problem
\begin{equation}\label{eq:PenProb2Intro}
	\phi_*:=\min \left\{ \phi(z):=f(z) + h(z)  : z \in \mathbb{R}^n \right \}
\end{equation}
where $h:\R^n \to (-\infty,\infty]$ is a proper lower semicontinuous convex function and
$f$ is a real-valued differentiable (possibly nonconvex) function with an $L$-Lipschitz continuous gradient
on a compact convex set containing
$\dom h $.
The N-SCO problem \eqref{eq:PenProb2Intro} has a wide range of real-world applications including support vector machine \cite{LanUniformly}, sparse PCA \cite{gu2014sparse}, matrix completion \cite{yao2017efficient}, and nonnegative matrix factorization \cite{gillis2014and,lee1999learning}.
ACG methods are widely used optimization approaches for solving N-SCO problems.
A critical issue
related to the practical performance of these methods lies on
the development of efficient stepsize selection strategies.

More specifically, a key step in ACG methods for solving
the N-SCO problem \eqref{eq:PenProb2Intro} is to compute an iterate $y_{k+1}$ as the solution of
a proximal subproblem of the form 
\begin{equation}\label{eq:update}
	y_{k+1}=y(\tx_k;M_k):=\argmin \left\lbrace  \ell_f(x;\tx_k) + h(x) + \frac{M_k}2 \|x-\tx_k\|^2 : x \in \R^n \right\rbrace 
\end{equation}
where $ \ell_f(x;\tx_k):=f(\tx_k) + \inner{\nabla f(\tx_k)}{x-\tx_k} $,
 $\tx_k$ is a convex combination of
$y_k$ and another auxiliary iterate $x_k$, and
$M_k$ is a positive scalar such that
\begin{equation}\label{ineq:descent}
{\cal C}(y(\tx_k;M_k);\tx_k) \le \tau M_k
\end{equation}
where $ \tau \in(0,1] $ and
\begin{equation}\label{eq:cal C}
{\cal C}(y;\tx) := \frac{2\left[ f(y)-\ell_f(y; \tx) \right] }{\|y-\tx\|^2}.
\end{equation}
It can be shown that the smaller the sequence $ \{M_k\} $ is, the faster
the convergence rate of the method becomes.
Hence, it is desirable to choose $M_k=\bar M_k$ where $\bar M_k$
is the smallest value of $M_k$
satisfying \eqref{ineq:descent}.
Since $\bar M_k$ is hard to compute,
a large class of ACG methods for solving either
convex or nonconvex SCO problems
simply computes
a scalar $M_k$ satisfying \eqref{ineq:descent} either
by setting it to be a  sufficiently large constant or by using a
line search procedure.
Works dealing with
ACG methods for solving
nonconvex SCO problems
based on this idea are discussed in ``Other related works" below.



In contrast to the
ACG methods
based on the above ideas,
the 
AC-ACG methods of \cite{liang2021average} do not require
the next iterate $y_{k+1}$ to satisfy \eqref{ineq:descent}.
They instead follow
the natural geometrical viewpoint of choosing
$M_k$, and hence
the local approximate model
in \eqref{eq:update}, by means of average curvature information.
More specifically,
the theoretical version of AC-ACG in \cite{liang2021average} computes
$y_{k+1}$ as in \eqref{eq:update}
with $M_k$
set to be  a positive multiple of the average of all observed curvatures $\tilde {\cal C}_0, \ldots,\tilde{\cal C}_{k-1}$, where $\tilde C_i:=\tilde {\cal C}(y_{i+1};\tx_i) $ for every $i$, and
\begin{equation}\label{eq:tC}
	\tilde {\cal C} (y;\tilde x):=  \max\left\lbrace 
	{\cal C}(y;\tx), {\cal L}(y;\tx)
	\right\rbrace, \quad 
	{\cal L}(y;\tx):= \frac{\|\nabla f(y) - \nabla f(\tilde x)\|}{\|y-\tx\|}.
\end{equation}
It is shown in Theorem 2.1 of \cite{liang2021average} that, for every
$k$, the theoretical version of AC-ACG generates a pair $ (\hat y_k,\hat v_k)$ satisfying
$ \hat v_k \in \nabla f(\hat y_k) + \partial h(\hat y_k)$ and $\|\hat v_k\|^2={\cal O}(M_k/k)$, and hence that its convergence rate
directly depends on the magnitude of $M_k$.
Paper \cite{liang2021average} also presents a practical aggressive variant of AC-ACG, which computes $M_k$ using the average of the ${\cal C}_i$'s, $i=0,\ldots,k-1$, where
$C_i:= {\cal C}(y_{i+1};\tx_i) $ for every $i$, instead of the usually much larger $\tilde C_i$'s.
Most likely due to smaller size of the generated sequence $\{M_k\}$,
the practical AC-ACG variant computationally
outperforms other ACG variants,
including its theoretical variant,
but its convergence rate analysis is left open in \cite{liang2021average}.

This paper presents the AC-FISTA method
for solving \eqref{eq:PenProb2Intro}
which is a FISTA-type variant of AC-ACG, 
and establishes a convergence rate for it
similar to the one described above but with
$M_k$ obtained in the same way as in the practical AC-ACG variant.
AC-FISTA has the following advantages compared to the
theoretical  AC-ACG variant.
It computes about half  the number of
composite resolvent evaluations as that performed by AC-ACG,
and hence its iterations are computationally cheaper.
It uses $C_k$ (in place of $\tilde C_k$) to compute
$M_k$, and hence generates 
a sequence of smaller curvature estimates  $\{M_k\}$.
(It is worth noting that,
even though this last property is established in the context
of AC-FISTA, it indirectly addresses
the aforementioned open question
of \cite{liang2021average} posed in the context of AC-ACG.)
\ Finally, computational results are presented in this paper
to demonstrate that AC-FISTA substantially outperforms
previous ACG variants as well as the theoretical and practical
AC-ACG variants,
both in terms of CPU time and computed  solution
quality.

{\bf Other related works.}  
The first convergence analysis of an ACG algorithm based on \eqref{ineq:descent}
for solving the N-SCO problem \eqref{eq:PenProb2Intro} under the same assumptions as in this paper appears in \cite{nonconv_lan16}.
Inspired by \cite{nonconv_lan16}, many papers have proposed other ACG variants based on \eqref{ineq:descent}.
Algorithms presented in
\cite{Paquette2017,Lin_Zhou_Liang_Varshney,liang2019fista,Yao_et.al.} choose $ M_k $ constant as in \cite{nonconv_lan16}, i.e.,
$ M_k = L/\tau $ for every $k$, where $ L$ is the Lipschitz constant of
$ \nabla f $ and $\tau\in(0,1]$.
Moreover, algorithms discussed in \cite{LanUniformly,Li_Lin2015,liang2019fista} use line search procedures to compute a relatively small scalar $ M_k $ satisfying \eqref{ineq:descent}.

In addition to the ACG methods mentioned above, it is worth discussing
other approaches for solving \eqref{eq:PenProb2Intro} that use an inexact proximal point method where each proximal subproblem
is constructed to be (possibly strongly) convex and hence solved by a convex ACG variant.
Papers
\cite{carmon2018accelerated,KongMeloMonteiro,paquette2018catalyst} describe a descent unaccelerated inexact 
proximal-type method that works with a large prox stepsize and approximately solves a proximal subproblem by an ACG variant.
Finally, \cite{jliang2018double} proposes an accelerated inexact proximal point method, which in each outer iteration performs an accelerated step with a large prox stepsize and 
follows the same way as in the algorithms presented in  \cite{carmon2018accelerated,KongMeloMonteiro} to solve a proximal subproblem.



\textbf{Basic definitions and notation.}
Let $\mathbb{R}$ and $\mathbb{R}_+$ denote the set of real numbers and the set of non-negative real numbers, respectively.
Let $\mathbb{R}^n$ denote the standard $n$-dimensional Euclidean 
space with  inner product and norm denoted by $\left\langle \cdot,\cdot\right\rangle $
and $\|\cdot\|$, respectively. 
The 
Frobenius norm in $ \R^{m\times n} $
is denoted by 
$ \|\cdot\|_F $.
The set of real symmetric positive semidefinite matrices in $ \R^{ n\times n} $ is denoted by $ {\cal S}^n_{+} $.
Let $\lceil \cdot \rceil$ denote the ceiling function.
The cardinality of a 
finite set $ A $ is denoted by $ |A| $. 
The indicator function $ I_X $ of a set $ X\subset \R^n $ is defined as $ I_X(z)=0 $ for every $ z\in X $, and $ I_X(z)=\infty $, otherwise.
The diameter of a compact set $ X\subset \R^n $ is $D_X  :=  \sup \{  \| z-\bar z \| : z, \bar z \in X \}$.
If $ X $ is a nonempty closed convex set, the orthogonal projection $ P_{X}: \R^n \rightarrow \R^n $ onto $ X $ is defined as 
\[
P_{X}(z):=\argmin_{\bar z\in X} \|\bar z-z\| \quad \forall z\in \R^n.
\]

Let $\Psi: \mathbb{R}^n\rightarrow (-\infty,+\infty]$ be given. The effective domain of $\Psi$ is denoted by
$\dom \Psi:=\{x \in \mathbb{R}^n: \Psi (x) <\infty\}$ and $\Psi$ is proper if $\dom \Psi \ne \emptyset$.
Moreover, a proper function $\Psi: \mathbb{R}^n\rightarrow (-\infty,+\infty]$ is $\mu$-strongly convex for some $\mu \ge 0$ if
	$$
	\Psi(\beta z+(1-\beta) \bar z)\leq \beta \Psi(z)+(1-\beta)\Psi(\bar z) - \frac{\beta(1-\beta) \mu}{2}\|z-\bar z\|^2
	$$
	for every $z, \bar z \in \dom \Psi$ and $\beta \in [0,1]$.
	Let $\partial \Psi (z)$ denote the subdifferential of $\Psi$ at $z \in \dom \Psi$.
If $\Psi$ is differentiable at $\bar z \in \mathbb{R}^n$, then its affine   approximation $\ell_\Psi(\cdot;\bar z)$ at $\bar z$ is defined as
$ \ell_\Psi(z;\bar z) :=  \Psi(\bar z) + \inner{\nabla \Psi(\bar z)}{z-\bar z} $ for every $ z \in \R^n $.
The set of all proper lower semi-continuous convex functions $\Psi:\mathbb{R}^n\rightarrow (-\infty,+\infty]$  is denoted by $\bConv{n}$.

\textbf{Organization of the paper.} 
Section~\ref{sec:alg+cmplx} consists of two subsections. Subsection~\ref{subsec:AC-FISTA} describes the assumptions made on the N-SCO problem and presents the AC-FISTA method for solving it.
It also describes the main result of the paper, which
establishes a convergence rate bound for AC-FISTA in terms of the average of observed curvatures.
Subsection~\ref{subsec:practice} discusses a special case of AC-FISTA which is quite efficient in practice, and
provides reasons behind its good performance.
Section~\ref{sec:proof} provides the proof of the main result
stated in Subsection~\ref{subsec:AC-FISTA}.
Section~\ref{sec:numerics} presents computational results demonstrating the efficiency of AC-FISTA.
Finally, Section~\ref{sec:conclusion} provides some
concluding remarks.

\section{AC-FISTA  and the main result}\label{sec:alg+cmplx}

This section consists of two subsections. The first one describes the assumptions made on the N-SCO problem \eqref{eq:PenProb2Intro} and presents the AC-FISTA method for solving it. It also presents
results describing the global convergence rate of
AC-FISTA in terms of the iteration count and some parameters
associated with AC-FISTA and the problem instance.
The second subsection discusses a special case of
AC-FISTA and the practical consequences of the above results
to this context.

\subsection{AC-FISTA and its theoretical guarantees}\label{subsec:AC-FISTA}

Throughout this paper, we consider the N-SCO problem \eqref{eq:PenProb2Intro} and make the following assumptions on it:
\begin{itemize}
	\item[(A1)] $h \in \bConv{n}$;  
	\item[(A2)] there exist scalar $L \ge 0 $ and
	a compact convex set $\Omega \supset {\cal H} := \dom h$
	such that $f$ is nonconvex and differentiable on $ \Omega $,
	and 
	\beq\label{ineq:curv}
	\|\nabla f(u)- \nabla f(u') \| \le L \|u-u'\| \qquad \forall u,u' \in \Omega.
	\eeq
\end{itemize}


We now make some remarks about the above assumptions.
First, it follows from (A1) and (A2) that
the set of optimal solutions $ X_* $ is nonempty and compact.
Second, if $L$ satisfies \eqref{ineq:curv} then the pair $(M,m)=(L,L)$ satisfies
\begin{equation}\label{ineq:upper}
	- \frac {m}2\|u-u'\|^2 \le  f(u)-\ell_f(u;u')\le\frac {M}2\|u-u'\|^2 \quad \forall u,u'\in \Omega.
\end{equation}

Throughout this paper,
$\bar L$ denotes the smallest $L$ satisfying \eqref{ineq:curv},
and
$\bar m$ (resp., $\bar M$)  denotes the smallest $m$ (resp., $M$)
satisfying the first (resp., second) inequality in \eqref{ineq:upper}.
Clearly, in view of (A2) and the second remark above, we have
$0 < \bar m\le \bar L$ and $0 \le \bar M \le \bar L$.

A necessary condition for $ y$ to be a local minimum of \eqref{eq:PenProb2Intro} is that $  y $ is a stationary
point of \eqref{eq:PenProb2Intro}, i.e.,
$0 \in \nabla f( y) + \partial h( y)$.
The goal of AC-FISTA described below is to find an approximate stationary point defined as follows.

\begin{definition}\label{def:approx pt}
	Given a tolerance $ \hat \rho>0 $, a pair $ (\hat y, \hat v) \in \R^n \times \R^n $ is called a $ \hat \rho $-approximate
	stationary point of $ \eqref{eq:PenProb2Intro} $ if
	it satisfies 
	$ \hat v \in \nabla f(\hat y) + \partial h(\hat y)$ and $ \|\hat v\|\le \hat \rho $.
\end{definition}

We are now ready to state AC-FISTA.

\noindent\rule[0.5ex]{1\columnwidth}{1pt}

AC-FISTA

\noindent\rule[0.5ex]{1\columnwidth}{1pt}

\begin{itemize}
	\item[0.] Let parameters $ \alpha,
	\gamma \in(0,1]$,
	 scalar $M$ such that $0.9 M \ge \bar M$,
	 tolerance $ \hat \rho>0 $,  and initial point $y_0 \in {\cal H}$ be given, and set $A_0 =0$, $x_0=y_0$, $M_0=\gamma M$ and $k=0$;
	\item[1.] compute
	\begin{equation}\label{eq:aktx}
	a_k = \frac{1+ \sqrt{1 + 4 M_{k}A_k}}{2M_{k}},\quad
	A_{k+1} = A_k + a_k, \quad
	\tx_k  =  \frac{A_ky_k+a_kx_k}{A_{k+1}};
	\end{equation}
	\item[2.] compute
	\begin{align}
y_{k+1}^g &=y(\tx_k;M_k), \quad C_k =  {\cal C}(y^g_{k+1};\tx_k), \label{eq:C} \\
	v_{k+1} &=  M_k(\tilde{x}_k - y^g_{k+1}) + \nabla f(y^g_{k+1}) - \nabla f(\tilde{x}_k) \label{eq:vk}
	\end{align}
	where $ y(\cdot;\cdot) $ and ${\cal C}(\cdot;\cdot)$ are as in \eqref{eq:update} and \eqref{eq:cal C}, respectively;
	if $\|v_{k+1} \| \le \hat \rho$ then output  $(\hat y,\hat v)=(y^g_{k+1},v_{k+1})$ and {\bf stop};
	\item[3.] 
	if $C_k\le 0.9M_k$,
	then compute
	\begin{equation}\label{eq:x+}
		x^g_{k+1} = P_\Omega\left( \frac{A_{k+1}}{a_k} y_{k+1}^g - \frac{A_k}{a_k} y_k \right),
	\end{equation}
	and set $ x_{k+1}=x_{k+1}^g $ and $\tilde y_{k+1}=y_{k+1}^g$;
	otherwise,  
    compute
	\begin{align}
		x_{k+1}^b & = \underset{u\in\R^n}{\mbox{argmin}} \left\{ a_k [\ell_f(u; \tilde{x}_k) + h(u)] + \frac{1}{2} \| u - x_k \|^2 \right\} \label{eq:xb},\\
		y_{k+1}^b  &=\frac{A_k y_k + a_k x_{k+1}^b}{A_{k+1}}, \label{eq:ty}
	\end{align}
	and set $ x_{k+1}=x_{k+1}^b $ and $ \tilde y_{k+1}=y_{k+1}^b $;
	\item[4.] choose $y_{k+1} \in {\cal H}$ such that $ \phi(y_{k+1}) \le \phi(\tilde y_{k+1}) $,
	compute
	\begin{align}
	M_{k+1} = \max \left\{ \gamma M\,,\, \frac{\sum_{j=0}^{k} C_j}{\alpha(k+1)} \right \} \label{eq:M},
	\end{align}
set $k \leftarrow k+1$, and go to step 1.
\end{itemize}

\noindent\rule[0.5ex]{1\columnwidth}{1pt}

The $k$-th iteration of
AC-FISTA is called good (resp., bad)
if the inequality
at the beginning of step 3,
which is identical to
\eqref{ineq:descent} with $\tau=0.9$ ,
is satisfied (resp., violated).
Moreover, for the sake of future reference,
we define the index sets for the good and bad iterations as
\begin{equation}\label{def:G}
	{\cal G} := \{k\ge0: C_k\le 0.9M_k\}, \quad {\cal B}:=\{k\ge0: C_k> 0.9M_k\},
\end{equation}
respectively.
For ACG methods that satisfy \eqref{ineq:descent} for every $k \ge 0$, it is well-known that the smaller the sequence $\{M_k\}$ is,
the faster their
practical performance is. One of the goals of
this paper is to show that this observation also holds for AC-FISTA, even though it does not satisfy \eqref{ineq:descent}
at every $k \ge 0$.
Hence, from the AC-FISTA point of view,
it is desirable to choose
$\alpha$ large, say $\alpha=0.5$,
and $\gamma$ small, say $\gamma=10^{-6}$,
since this forces $M_k$
in \eqref{eq:M}
to be small.
It turns out that this is the AC-FISTA implemented
in our benchmark of Section \ref{sec:numerics} and we refer to
it as the $(0.5,10^{-6})$-AC-FISTA.

The following paragraphs give some relevant comments about AC-FISTA.

It is shown below in Theorem \ref{thm:main}(a) that the pair $ (y_{k+1}^g, v_{k+1}) $ satisfies the inclusion in Definition \ref{def:approx pt} for every $ k\ge 0 $. Hence, if  the termination criterion $ \|v_{k+1}\|\le \hat \rho $ in step 2
is satisfied, then AC-FISTA terminates with  a $ \hat \rho $-approximate
stationary point of $ \eqref{eq:PenProb2Intro} $.
The first two identities in \eqref{eq:aktx} imply that
\begin{equation}\label{eq:rel}
	A_{k+1} = M_k a_k^2.
\end{equation}
It follows from step 0 of AC-FISTA, \eqref{eq:M}, \eqref{eq:cal C}, \eqref{ineq:upper} with $(m,M)=(\bar m, \bar M)$, and the definition of $C_k$
in \eqref{eq:C}, that
\beq \label{eq:curv-obs}
M_k \ge \gamma M, \quad C_k \in [-\bar m, \bar M], \quad  \forall k \ge 0.
\eeq

Two popular rules for choosing $y_{k+1}$
in step 4 of AC-FISTA are:
i) $y_{k+1}=\tilde y_{k+1}$ for all $k\ge 0$;
and
ii) $y_{k+1}$ such that $\phi(y_{k+1}) = \min\{ \phi(y_k), \phi(\tilde y_{k+1}) \}$ for all $k\ge 0$.
If rule (i) is chosen,
then an iteration of
AC-FISTA works as follows.
Given a pair $(x_k,y_k)$,
the $k$-th iteration
sets 
$(x_{k+1},y_{k+1})$
as being the pair
$(x_{k+1}^g,y_{k+1}^g)$ obtained in
\eqref{eq:x+} and \eqref{eq:C} if  $k \in {\cal G}$,
or the pair $(x_{k+1}^b,y_{k+1}^b)$ obtained in
\eqref{eq:xb} and \eqref{eq:ty} if $k \in {\cal B}$.
The condition on $y_{k+1}$
in step 4
simply relaxes rule (i),
and allows
AC-FISTA  to include as
special case
the monotone 
(i.e., satisfying $\phi(y_{k+1})\le \phi(y_k)$ for all $k$) variant
in which,
in place of (i),
rule (ii) is used instead.


Following the same notation of this paper, Subsection 3.1 of \cite{liang2021average} reviews three rules for performing an ACG iteration
which have roots
in works dealing with the
convex version of SCO \eqref{eq:PenProb2Intro}.
They are referred there
to as FISTA rule, AT rule, and LLM rule.
Under the assumption that
$\Omega = \R^n$,
the
two AC-FISTA iterations
(i.e., good and bad)
can be interpreted in terms
of the three rules above as follows:
a good (resp., bad) iteration of AC-FISTA
performs an ACG iteration based on the FISTA (resp., AT) rule. 
Since the test to
decide the type of
iteration (i.e., good or bad) to perform
depends on
$y_{k+1}^g$, this point
needs to be computed
prior to a bad iteration
(even though the iteration itself
does not use it).



We now comment on the computational effort of
an AC-FISTA iteration. A good iteration computes only one resolvent
evaluation of $ \partial h $,
while a bad one computes two composite
resolvent evaluations (one in \eqref{eq:C}
to compute $y_{k+1}^g$ and
another in \eqref{eq:xb} to compute $x_{k+1}^b$).
Since $\Omega$ is usually chosen so that
the projection onto $\Omega$ in \eqref{eq:x+}
is considerably cheaper than
a composite resolvent evaluation and
the majority of iterations performed by AC-FISTA is assumed to be
good ones (see Condition A below), it follows
that the average number of resolvent evaluations per iteration 
is close to 1.

We now state the main result of the paper which describes
how fast one of the iterates $y_1^g,\ldots,y_k^g$ approaches the
stationary condition $0 \in \nabla f(y)+\partial h(y)$.
Its main conclusion assumes the following condition.

\vspace*{.05in}

\noindent
{\bf Condition A:}
There exist $ k_0 \in \NN_+ $ such that $ |{\cal B}_k| \le k/3 $ for every $ k\ge k_0 $
where
\begin{equation}
 {\cal B}_k:=\{i \in {\cal B} : i\le k-1 \} \qquad
	\forall k \ge 1.  \label{def:Gk}
\end{equation}

It is worth noting that the factor $1/3$ in Condition A
is not particularly important for our analysis to hold.
Even though this factor can be replaced by
any scalar less than one, we have chosen a specific value for
it in order to keep the number of constants used in our
analysis small.

We now discuss some choices of $(\alpha,\gamma)$ which guarantee that Condition A holds.
First,
step 0 of AC-FISTA and \eqref{eq:curv-obs} imply that
$C_k \le \bar M \le 0.9M \le 0.9 M_k/\gamma$ for every $k \ge 0$.
Thus,
if $\gamma=1$ (and $\alpha \in (0,1]$ is arbitrary) then
every iteration of
AC-FISTA is good,
and
Condition A trivially holds with $k_0=0$.
Second,
it is shown in Lemma 4.5 of \cite{liang2021average} that
Condition A holds with $k_0=12$
whenever $\alpha, \gamma \in (0,1]$ are chosen so that
\begin{equation}\label{eq:alpha}
	\alpha \le \frac{0.9}{8}\left( 1+\frac{1}{0.9\gamma}\right)^{-1}.
\end{equation}
Rule \eqref{eq:alpha} for choosing
$(\alpha,\gamma)$ results in $\alpha={\cal O}(\gamma)$ so that
a small choice of $\gamma$ implies that $\alpha$ is also small.
Hence, it excludes the
practical choice
$(\alpha,\gamma)=(0.5,10^{-6})$
mentioned in the paragraph following
AC-FISTA.
However, since the proof of Lemma 4.5 of \cite{liang2021average}
is based on quite conservative bounds,
Subsection \ref{subsec:practice} below
reexamines its proof
and gives
strong evidence
(validated by our computational results)
that Condition A holds
for our practical choice of $(\alpha,\gamma)=(0.5,10^{-6})$.

Instead of specifying
values for
$(\alpha,\gamma)$, the two main results below simply assume that
Condition A holds, 
and establishes
a global convergence
rate and iteration-complexity for AC-FISTA.




\begin{theorem}\label{thm:main}
	Define the harmonic mean of the sequence $ \{M_i\} $
	and the average of the curvature
	sequence $ \{ {\cal L}(y^g_{i+1}; \tx_i ) \} $ defined in
	\eqref{eq:tC} as
	\begin{equation}
		M_k^{hm}:=\frac{k}{\sum_{i=0}^{k-1} \frac{1}{M_i}}, \quad
		L_k^{avg}:= \frac{1}{k}\sum_{i=0}^{k-1} {\cal L}(y^g_{i+1}; \tx_i) \label{eq:L},
	\end{equation}
	respectively, and let
	\begin{equation}\label{def:theta}
		\theta_{k}:= \frac{M_k}{M_k^{hm}}, \quad \tau_k:= \frac{L_k^{avg}}{M_k}.
	\end{equation}
	Then, the following statements hold:
	\begin{itemize}
		\item[(a)]
		for every $k \ge 1 $, we have
		$v_k \in \nabla f(y_{k}^g)+\partial h(y_{k}^g)$;
		\item[(b)] if Condition A holds, then for every $ k\ge \max\{12,k_0\} $,
		\begin{equation}\label{eq:bound}
			\min_{1 \le i \le k} \|v_i\| = {\cal O}\left( \left( 1 +  \tau_k \right) \left[ \frac{M_k d_0}{k^{3/2}} + \left( \sqrt{\bar M} + \sqrt{\bar m } \right) \frac{\sqrt{M_k \theta_k} D_\Omega}{k} + \frac{\sqrt{\bar m M_k \theta_k} D_{\cal H} }{\sqrt{k}} \right]\right)
		\end{equation}
		where $d_0$ denotes the distance of the initial point $x_0$ to the set of optimal solutions of \eqref{eq:PenProb2Intro}, and $ D_\Omega $ and $ D_{\cal H} $ denote the diameters of $ \Omega $ and $ {\cal H} $, respectively;
    \item[(c)] for every $k\ge 1$, we have
    \begin{equation}\label{ineq:tau}
		    M_k= {\cal O}\left(\frac{M}{\alpha}\right), \quad \frac{k-1}{2k} \le \theta_k \le \frac{M_k}{\gamma M}, \quad \tau_k \le \frac{\bar L}{\gamma M};
		\end{equation}
  as a consequence, $\theta_k={\cal O}( 1/(\alpha \gamma))$.
	\end{itemize}
\end{theorem}

	
	
	 The corollary below describes the
	worst-case behavior of AC-FISTA under the assumption that Condition A holds.

%
%
%
%
%


\begin{corollary}\label{cor:bound}
If Condition A holds, then for every $ k\ge \max\{12,k_0\} $,
	\begin{equation}\label{eq:bound2}
			\min_{1 \le i \le k} \|v_i\| = {\cal O}\left( \left( 1 +  \frac{\bar L}{\gamma M} \right) \left[ \frac{M d_0}{\alpha k^{3/2}} + \left( \sqrt{\bar M} + \sqrt{\bar m }\right) \frac{\sqrt{M} D_\Omega}{\alpha\sqrt{\gamma} k} + \frac{\sqrt{\bar m M} D_{\cal H} }{\alpha \sqrt{\gamma} \sqrt{k}} \right]\right).
		\end{equation}
  As a consequence, the iteration-complexity to find a $\hat \rho$-approximate stationary point of \eqref{eq:PenProb2Intro} is
  \[
  {\cal O}\left( \left( 1 +  \frac{\bar L}{\gamma M} \right)^{2/3} \left(\frac{M d_0}{\alpha \hat \rho}\right)^{2/3} + \left( 1 +  \frac{\bar L}{\gamma M} \right) \frac{( \sqrt{\bar M} + \sqrt{\bar m })\sqrt{M}D_\Omega}{\alpha\sqrt{\gamma} \hat \rho} + \left( 1 +  \frac{\bar L}{\gamma M} \right)^2 \frac{\bar m M D_{\cal H}^2}{\alpha^2\gamma\hat \rho^2}\right).
  \]
\end{corollary}

\begin{proof}
The first conclusion immediately follows from Theorem \ref{thm:main}(b) and (c).  The second one follows from \eqref{eq:bound2} and Definition \ref{def:approx pt}.
\end{proof}



We finally make a comment about the
dependence of the convergence rate bound \eqref{eq:bound2} in terms of $ \alpha $ and $\gamma$ only.
If $M$ is chosen so as to satisfy
$M \ge \bar L/0.9$, then the
bottleneck term in the worst-case convergence rate bound \eqref{eq:bound2} becomes
${\cal O}(\sqrt{\bar m M} D_{\cal H}/(\alpha \gamma^{3/2}\sqrt{k}))$,
which is equal to $1/(\alpha \gamma^{3/2})$ times
the convergence rate bounds of some other ACG variants
derived in the literature
(see e.g.,\ \cite{nonconv_lan16,liang2019fista}).
Hence, the above two convergence rate bounds
are similar whenever $\alpha$ and
$\gamma$ are both close to one.



\subsection{A practical AC-FISTA}\label{subsec:practice}

This subsection discusses in more details the AC-FISTA
with $(\alpha,\gamma)=(0.5,10^{-6})$,
i.e., the
$(0.5,10^{-6})$-AC-FISTA
as defined in the paragraph immediately after AC-FISTA.


Although the dependence of the dominant term in \eqref{eq:bound2}
with respect to $\alpha$ and $\gamma$, i.e., ${\cal O}(1/(\alpha\gamma^{3/2}))$,
is large for $(0.5,10^{-6})$-AC-FISTA, it should be noted that
this dependence factor was obtained using the conservative estimates
in 
\eqref{ineq:tau}. In practice, the quantity
$\theta_k$, although sometimes
		initially large,
		quickly approaches one and stays close to
		one thereafter, and $\tau_k$
		never exceeds $4$ and, in many cases,
		is within the interval $[0,3]$ 
		(see Tables \ref{tab:SVM}, \ref{tab:QP1}, \ref{tab:QP2}, \ref{tab:MC1} and \ref{tab:MC2}).
		We can then conclude that, in terms of
		$\gamma$ and $\alpha$,
		both $\theta_k$ and $\tau_k$ are ${\cal O}(1)$, instead of ${\cal O}(1/(\alpha\gamma))$ and ${\cal O}(1/\gamma)$
		as in Theorem \ref{thm:main}(c).
		Second, under the (often observed)
		condition that $\theta_k$ and $\tau_k$ are ${\cal O}(1)$, the convergence rate bound
		reduces to
\[
\min_{1 \le i \le k} \|v_i\| = {\cal O}\left(  \frac{M_k d_0}{k^{3/2}} + \left( \sqrt{\bar M} + \sqrt{\bar m} \right) \frac{\sqrt{M_k} D_\Omega}{k} + \frac{\sqrt{\bar m M_k} D_{\cal H} }{\sqrt{k}} \right),
\]
and hence does not depend on $\gamma^{-1}=10^6$.
Third, the latter bound clearly shows that
the practical behavior of this $(0.5,10^{-6})$-AC-FISTA improves
as the ratio  $M_k/M$ becomes small, which is
what has been observed in our computational
experiments.

We will now argue that in practice  the
$(0.5,10^{-6})$-AC-FISTA is likely to
satisfy Condition A.
Indeed, letting
\[
 \eta_k:=\frac{\sum_{i \in {\cal B}_k } C_i }
 {\sum_{i \in {\cal B}_k, \, i \le |{\cal B}_k|/2 } C_i }
 \]
 where ${\cal B}_k$ is as in \eqref{def:Gk},
 and examining the proof of Lemma 4.5 of \cite{liang2021average}, it can be easily seen that
 \[
 k \alpha  \eta_k \ge  \frac{0.9|{\cal B}_k|}2.
 \]
 Lemma 4.5 of \cite{liang2021average} then uses
 \eqref{eq:M}, and the facts that $ M \ge \bar M \ge C_i$ and $C_i> 0.9M_i$ for $i\in {\cal B}_k$, to conclude that
 $\eta_k$ is bounded by $1+1/(0.9\gamma)$.
 However, such bound on $\eta_k$ is quite conservative and
it is observed computationally that
 $\eta_k$ quickly approaches two and remains close to two
 thereafter.
 Hence, it follows from the latter observation and
 the above inequality that
 any choice of  $\alpha$ in $(0,1/8]$ makes Condition A
 very likely to hold in practice.
 Although our choice of $\alpha=0.5$ does not lie in this
 (still conservative) range, we have observed that it works quite
 well in our computational experiments.

\subsection{Comparison with the AC-ACG method of \cite{liang2021average}}

We start by giving an overview of 
the AC-ACG method of \cite{liang2021average}
using the description of
AC-FISTA in Subsection \ref{subsec:AC-FISTA}.
More specifically,
AC-ACG is similar to
the variant of AC-FISTA
where $y_{k+1}=\ty_{k+1}$
for every $k \ge 0$ 
but differs in the following
two aspects:
\begin{itemize}
    \item
    it sets
    $x_{k+1}$ to  the right-hand side of \eqref{eq:xb}
    regardless of whether the iteration is good or bad;
    \item
    it computes $M_{k+1}$ as
    in \eqref{eq:M} but with $C_j$ replaced by $\tilde {\cal C}(y^g_{j+1};\tx_j)$ where $\tilde C(\cdot;\cdot)$ is as in \eqref{eq:tC}.
\end{itemize}
Another not-so-crucial difference is that ACG chooses
the parameters $\alpha,\gamma \in (0,1]$ so that
\eqref{eq:alpha} holds as equality
while AC-FISTA 
discards the latter condition
on $\alpha$ and $\gamma$ and
instead
simply makes the weaker assumption that
Condition A holds
(see the third remark in the paragraph containing \eqref{eq:alpha}).

In terms of the three ACG rules described in Subsection 3.1 of \cite{liang2021average}, a good (resp., bad) iteration of AC-ACG performs an ACG iteration based on the LLM (resp., AT) rule.
Hence, while a good iteration
AC-ACG uses the
LLM rule, the one
for AC-FISTA uses the
FISTA rule.





From a computational point of view, while AC-ACG always performs two resolvent
evaluations at every iteration, AC-FISTA performs
one resolvent evaluation in a good iteration
and two resolvent evaluations in a bad one.
Since in practice most of
the iterations of AC-FISTA are good,
its average cost per iteration
is relatively lower
than that of
AC-ACG.




\section{Proof of Theorem \ref{thm:main}}\label{sec:proof}

 We start by providing a straightforward technical result
 which is then used 
 to outline our analysis in this section.

\begin{lemma}\label{lem:vk}
    For every $k \ge 1 $, we have
		$v_k \in \nabla f(y_{k}^g)+\partial h(y_{k}^g)$.
\end{lemma}
\begin{proof}
    The inclusion follows from the optimality condition of \eqref{eq:update}, and the definitions of $ y_{k+1}^g $ and $ v_{k+1} $ in \eqref{eq:C} and \eqref{eq:vk}, respectively.
\end{proof}

The most technical and difficult part of Theorem \ref{thm:main}
is its statement (b) where a convergence rate bound on $\min_{1\le i \le k}\|v_i\|$ is claimed.
A rough outline of the proof of this statement is as follows.
First, Lemmas \ref{lem:gamma}-\ref{lem:sum} are used to prove that
\begin{equation}\label{eq:bnd1}
    \sum_{i\in {\cal G}_k}\left( A_{i+1}M_i\|y_{i+1}^g-\tx_i\|^2\right) 
		= {\cal O}\left( d_0^2 +  \frac{\bar m + \bar M}{M_k^{hm}}  D_\Omega^2k + \frac{\bar m}{M_k^{hm}} D_{\cal H} ^2 k^2 \right)
\end{equation}
where ${\cal G}_k=\{0,\ldots,k-1\}\setminus {\cal B}_k$.
Next, using Condition A and some nontrivial technical results,
namely, Lemmas \ref{lem:Gk}-\ref{lem:new}, it is shown within the proof of Theorem \ref{thm:main}(b) that
\begin{equation}\label{ineq:minv}
    \min_{1 \le i \le k} \|v_i\| 
		= {\cal O}\left(\frac{M_k + L_k^{avg}}{k^{3/2}}  \left( \sum_{ i \in {\cal G}_k} A_{i+1}M_i\|y_{i+1}^g-\tx_i\|^2\right)^{1/2}\right).
\end{equation}
A direct combination of the above two claims
then immediately gives us
Theorem \ref{thm:main}(b).
The proof of Theorem \ref{thm:main}(c) does not require any technical result.

The first lemma below states
a few basic properties of AC-FISTA.

%
%
%
%
%
%
%

\begin{lemma}\label{lem:gamma}
	For every $k \ge 0$, we define
	\begin{align}
		& \tilde{\gamma}_k(u) :=  \ell_f(u; \tilde{x}_k) + h(u), \label{def1} \\
		& \gamma_k(u) := \tilde{\gamma}_k(y_{k+1}^g) + M_k\langle \tilde{x}_k - y_{k+1}^g, u - y_{k+1}^g \rangle.  \label{def2}
	\end{align}
	Then the following statements hold for every $k \ge 0$:
	\begin{itemize}
		\item[(a)] $\gamma_k$ minorizes $\tilde{\gamma}_k$, $\tilde{\gamma}_k(y_{k+1}^g) = \gamma_k(y_{k+1}^g)$,
		\[
			\min_u \left\{ \tilde{\gamma}_k(u) + \frac{M_k}{2}\| u - \tilde{x}_k \|^2 \right\}  = 
			\min_u \left\{ \gamma_k(u) + \frac{M_k}{2}\| u - \tilde{x}_k \|^2 \right\},
		\]
		and these minimization problems have $y_{k+1}^g$ as 
		unique optimal solution; 
		\item[(b)] for every $u  \in {\cal H} $, $ \tilde{\gamma}_k(u) -  \phi(u) \le \bar m  \| u - \tilde{x}_k \|^2/2; $
		\item[(c)]
		$
		x_{k+1}^g = {\mbox{argmin}} \left\{ a_k \gamma_k(u) + \| u - x_k \|^2/2: u\in\Omega \right\};
		$
		\item[(d)]
		$ \{x_{k}^b\} $, $\{y_k\}$, $\{y_{k}^g\}$ and $\{\ty_{k}\}$ are contained in $ {\cal H} $, while $\{x_k^g\}$, $\{x_k\}$ and $\{\tx_k\}$ lie
		in $\Omega$;
		\item[(e)] for every $u \in {\cal H} $, we have 
		\[
		A_k \| y_k - \tilde{x}_k \|^2 + a_k \| u - \tilde{x}_k\|^2 \le \frac{1}{M_k} D_\Omega^2 + a_k  D_{\cal H}^2.
		\]
	\end{itemize}
\end{lemma}


\begin{proof}
	(a) This statement follows from Lemma 2.2(a) of \cite{liang2019fista} with $ (\kappa_0, \lam, y_{k+1})=(0, 1/M_k, y_{k+1}^g) $.
	
	(b) This statement immediately follows from the first inequality in \eqref{ineq:upper} and the definition of $\tilde{\gamma}_k(u)$ in (\ref{def1}). 
	
	(c) It follows from the definitions of $ \tx_k $ and $\gamma_k$ in \eqref{eq:aktx} and \eqref{def2}, respectively, and relation \eqref{eq:rel} that the (unique) global minimizer of the function $ a_k \gamma_k(u) +  \| u - x_k \|^2/2$ over $\R^n$ is
    \[
    x_k + a_kM_k(y_{k+1}^g-\tilde x_k)=x_k + \frac{A_{k+1}}{a_k} \left(y_{k+1}^g - \frac{A_k y_k +a_k x_k}{A_{k+1}}\right)=\frac{A_{k+1}}{a_k} y_{k+1}^g - \frac{A_k}{a_k} y_k.
    \]
    This observation and the definition of $x_{k+1}^g$ in \eqref{eq:x+} then imply that the conclusion of (c) holds.

	(d) First, it is by definition that $ \{y_k\} $ is contained $ {\cal H} $. In view of \eqref{eq:update} (resp., \eqref{eq:xb}), it is clear that the sequence $ \{y_{k}^g\} $ (resp., $ \{x_{k}^b\}$) is contained in $ {\cal H}  $.
	Hence, using the fact that $ y_0\in {\cal H} $ (see step 0 of AC-FISTA), \eqref{eq:ty} and the convexity of ${\cal H} $, we easily see that $\{y_{k}^b\} \subset {\cal H} $ and hence $\{\ty_{k}\} \subset {\cal H} $. 
	It is also easy to see that $ \{x_k^g\} \subset \Omega $ from its definition in \eqref{eq:x+}. Hence, it follows from the fact that $\{x_k^b\} \subset {\cal H} \subset \Omega$ and step 0 of AC-FISTA that $\{x_k\}$ lie in $ \Omega $.
	Finally, $\{\tx_k\} \subset \Omega $ follows from the third identity in \eqref{eq:aktx} and the convexity of $\Omega$.
	
	(e) It is easy to see that for every $ x, y\in \R^n $ and $ a, A \in \R_+ $, 
	\[
	A\|y\|^2+a\|x\|^2=(A+a)\left\| \frac{Ay+ax}{A+a}\right\|^2 + \frac{Aa}{A+a}\|y-x\|^2.
	\]
	Using the above identity with $ x = u-\tx_k $, $ y=y_k-\tx_k $, $ a=a_k $ and $ A=A_k $, and the second and third identities in \eqref{eq:aktx}, we have
	\[
	A_k \| y_k - \tilde{x}_k \|^2 + a_k \| u - \tilde{x}_k\|^2  
		=  \frac{a_k^2}{A_{k+1}} \|u-x_k\|^2 + \frac{A_ka_k}{A_{k+1}}\|u-y_k\|^2.
	\]
	This statement now follows from the above inequality, statement (d), the definitions of $ D_\Omega $ and $ D_{\cal H} $, and relation \eqref{eq:rel}.
\end{proof}

The next result introduces a crucial potential function and
provides an important recursive formula based on it.


\begin{lemma}\label{lem:tech2}
	For every $ u\in{\cal H}  $ and $ k\ge 0 $, we have
		\begin{align}
		\frac{ M_k-F_k}{2}A_{k+1}\|\ty_{k+1}-\tx_k\|^2
		\le \eta_k(u)- \eta_{k+1}(u)+\frac{\bar m}{2} \left( \frac{1}{M_k} D_\Omega^2 + a_k  D_{\cal H} ^2\right) \label{ineq:good} 
		\end{align}
	where $M_k$ is as in \eqref{eq:M}, $ F_k:= {\cal C}(\ty_{k+1};\tx_k) $ and
	\begin{equation} \label{eq:eta}
	\eta_k(u):=A_k[\phi(y_k)-\phi(u)]+\frac12\|u-x_k\|^2.
	\end{equation}
\end{lemma}
\begin{proof}
	We first note that in order to prove the lemma, it suffices to show 
	\begin{equation}\label{ineq:goal}
		\frac{ M_k-F_k}{2}A_{k+1}\|\ty_{k+1}-\tx_k\|^2-\eta_k(u)+ \eta_{k+1}(u)
	\le A_k(\tilde \gamma_k(y_k)-\phi(y_k))+a_k(\tilde \gamma_k(u)-\phi(u)).
	\end{equation}
	Indeed, it follows from the above inequality and Lemma \ref{lem:gamma}(b) that 
	\[
		\frac{ M_k-F_k}{2}A_{k+1}\|\ty_{k+1}-\tx_k\|^2-\eta_k(u)+ \eta_{k+1}(u)
		\le  \frac{\bar m}{2}\left( A_k \| y_k - \tx_k \|^2 + a_k\| u - \tx_k\|^2\right),
	\]
	which, together with Lemma \ref{lem:gamma}(e), then immediately implies \eqref{ineq:good}.

	We now prove \eqref{ineq:goal} holds for $ k \in {\cal G} $.	
	Let $k \in {\cal G}$ and $u \in {\cal H}  $ be given.
	Noting that $ x_{k+1}=x_{k+1}^g $,
	and using Lemma \ref{lem:gamma}(c), relations \eqref{eq:aktx} and \eqref{eq:rel}, and the fact that $a_k\gamma_k + \|\cdot-x_k\|^2/2$ is $1$-strongly convex, we conclude that
	\begin{align*}
	A_k\gamma_k(y_k)&+a_k\gamma_k(u)+\frac{1}{2}\|u-x_k\|^2-\frac{1}{2}\|u-x_{k+1}\|^2 
	\ge A_k\gamma_k(y_k)+a_k\gamma_k(x_{k+1})+\frac{1}{2}\|x_{k+1}-x_k\|^2 \nn \\
	&\ge A_{k+1}\gamma_k(\hat y_{k+1})+\frac{1}{2}\frac{A_{k+1}^2}{a_k^2}\|\hat y_{k+1}-\tx_k\|^2 
	= A_{k+1}\left[ \gamma_k(\hat y_{k+1})+\frac{ M_k}{2}\|\hat y_{k+1}-\tx_k\|^2\right]
	\end{align*}
	where $ \hat y_{k+1} = (A_ky_k+a_kx_{k+1})/A_{k+1} $.
	It follows from Lemma \ref{lem:gamma}(a) and the fact that $ \ty_{k+1}= y_{k+1}^g $ for every $ k\in {\cal G} $ that
\begin{align*}
	\gamma_k(\hat y_{k+1})&+\frac{ M_k}{2}\|\hat y_{k+1}-\tx_k\|^2
	\ge \gamma_k(\ty_{k+1})+\frac{ M_k}{2}\|\ty_{k+1}-\tx_k\|^2\\
	&= \tilde \gamma_k(\ty_{k+1})+\frac{ M_k}{2}\|\ty_{k+1}-\tx_k\|^2
	 = \phi(\ty_{k+1})+\frac{ M_k-F_k}{2}\|\ty_{k+1}-\tx_k\|^2\\
& \ge \phi(y_{k+1})+\frac{ M_k-F_k}{2}\|\ty_{k+1}-\tx_k\|^2.
\end{align*}
where the last identity is due to the definitions of $ F_k $ and $ \tilde \gamma_k $ in \eqref{def1}, and the last inequality is due to the fact that $ \phi(y_{k+1}) \le \phi(\ty_{k+1}) $.
Using the above two inequalities and the definition of $ \eta_k $ in \eqref{eq:eta}, we have
\[
\frac{ M_k-F_k}{2}A_{k+1}\|\ty_{k+1}-\tx_k\|^2-\eta_k(u)+ \eta_{k+1}(u)
\le A_k[\gamma_k(y_k)-\phi(y_k)]+a_k[\gamma_k(u)-\phi(u)],
\]
which together with the fact that $ \gamma_k \le \tilde \gamma_k $ (see Lemma \ref{lem:gamma}(a)) implies that \eqref{ineq:goal} holds.

We finally prove \eqref{ineq:goal} holds for $ k\in {\cal B} $. Let $k \in {\cal B}$ and $u \in {\cal H} $ be given. 
Noting that $ x_{k+1}=x_{k+1}^b $ and $ \ty_{k+1}=y_{k+1}^b $,
and using the definitions of $ \tilde \gamma_k $, $ x_{k+1}^b $, $ \ty_{k+1} $ and $ \tx_k $ in \eqref{def1}, \eqref{eq:xb}, \eqref{eq:ty} and \eqref{eq:aktx}, respectively, the fact that $a_k \tilde \gamma_k + \|\cdot-x_k\|^2/2$ is $1$-strongly convex, and relation \eqref{eq:rel}, we conclude that
\begin{align*}
A_k \tilde \gamma_k(y_k)&+a_k \tilde \gamma_k(u)+\frac{1}{2}\|u-x_k\|^2-\frac{1}{2}\|u-x_{k+1}\|^2 
\ge A_k \tilde \gamma_k(y_k)+a_k \tilde \gamma_k(x_{k+1}^b)+\frac{1}{2}\|x_{k+1}^b-x_k\|^2  \\
&\ge A_{k+1} \tilde \gamma_k(y_{k+1}^b)+\frac{1}{2}\frac{A_{k+1}^2}{a_k^2}\|y_{k+1}^b-\tx_k\|^2 
= A_{k+1}\left[ \tilde \gamma_k(\ty_{k+1})+\frac{ M_k}{2}\|\ty_{k+1}-\tx_k\|^2\right]\\
&= A_{k+1}\left[ \phi(\ty_{k+1})+\frac{ M_k-F_k}{2}\|\ty_{k+1}-\tx_k\|^2 \right] 
 \ge A_{k+1}\left[ \phi(y_{k+1})+\frac{ M_k-F_k}{2}\|\ty_{k+1}-\tx_k\|^2 \right] 
\end{align*}
where the last identity is due to the definitions of $ F_k $ and $ \tilde \gamma_k $ in \eqref{def1}, and the last inequality is due to the fact that $ \phi(y_{k+1}) \le \phi(\ty_{k+1}) $.
	Using the above inequality and the definition of $ \eta_k $ in \eqref{eq:eta}, and rearranging the terms, we obtain \eqref{ineq:goal}.
\end{proof}

The following result discusses the consequences of Lemma \ref{lem:tech2}
when $k$ is a good iteration and also when $k$ is a bad one.

\begin{lemma}\label{lem:basic}
	The following statements hold for every $ u\in{\cal H}  $:
	\begin{itemize}
		\item[(a)] if $k \in {\cal G}$ then
		\begin{equation}\label{ineq:case1}
		\frac{1}{20} A_{k+1}M_k \|\ty_{k+1} - \tx_k\|^2\le \eta_k(u)- \eta_{k+1}(u)+\frac{\bar m}{2} \left( \frac{1}{M_k} D_\Omega^2 + a_k  D_{\cal H} ^2\right);
		\end{equation}
		\item[(b)] if $k \in {\cal B}$ then
		\begin{equation}\label{ineq:case2}
		0\le \eta_k(u)- \eta_{k+1}(u)+\frac{\bar m}{2} \left( \frac{1}{M_k} D_\Omega^2 + a_k  D_{\cal H} ^2\right) + \frac{\bar M}{2M_k}D_\Omega^2.
		\end{equation}
	\end{itemize}
\end{lemma}
\begin{proof}
	(a) Let $ k\in{\cal G}$ be given.
	It is easy to see that $ F_k \le 0.9M_k $ due to the fact that $ F_k=C_k $ when $ k\in {\cal G} $ and \eqref{def:G}, and hence \eqref{ineq:case1}  immediately follows from this observation and \eqref{ineq:good}.

	(b) Let $ k\in{\cal B} $ be given.
	Noting that $ x_{k+1}=x_{k+1}^b $ and $ \ty_{k+1}=y_{k+1}^b $,
	and using relations \eqref{ineq:good} and \eqref{eq:rel}, and the definitions of $ y_{k+1}^b $ and $ \tx_k $ in \eqref{eq:ty} and \eqref{eq:aktx}, respectively,
we conclude that
	\[
	\begin{aligned}
      \eta_k(u) &- \eta_{k+1}(u)+\frac{\bar m}{2} \left( \frac{1}{M_k} D_\Omega^2 + a_k  D_{\cal H} ^2\right) \ge
      \frac{M_k-F_k}{2}A_{k+1}\|\ty_{k+1}-\tx_k\|^2 \\
     &= \frac{M_k-F_k}{2}A_{k+1}\left \| \frac{A_ky_k+a_k x_{k+1}}{A_{k+1}} -  \frac{A_ky_k+a_k x_{k}}{A_{k+1}} \right \|^2 
= \frac{(M_k-F_k) a_k^2}{2A_{k+1} }\|x_{k+1}-x_k\|^2 \\
&= \frac12 \left(1 - \frac{F_k}{M_k} \right)
\|x_{k+1}-x_k\|^2
\ge \frac12 \left( 1 - \frac{\bar M}{M_k} \right) \|x_{k+1}-x_k\|^2
	\end{aligned}
	\]
	where the last inequality is due to the fact that $ F_k \le \bar M $,
    and hence that \eqref{ineq:case2} holds  in view of Lemma \ref{lem:gamma}(d).
	\end{proof}
	
	We now state a technical result which will be used
	to derive a consequence of Lemma \ref{lem:basic}.

\begin{lemma}\label{estimates}
	The sequences $ \{A_k\} $ and $ \{M_k\} $ generated by AC-FISTA satisfy the following statements:
	\begin{itemize}
		\item[a)] for every $ k\ge 1 $, we have
		$A_k M_k^{hm} \le k^2$;
		\item[b)] 
		for every $ k\ge 4 $, we have $ A_k M_k \ge k^2/12 $;
		\item[c)]
		for every $i \in \{1,\ldots,k\}$, we have
	 $ i M_i \le k M_k $.
	\end{itemize}
\end{lemma}

\begin{proof}
	a) This statement follow from  Lemma 4.7  of \cite{liang2021average}
	and the definition of $ M_k^{hm} $ in \eqref{eq:L}.
	
	b) This statement can be proved by following an argument similar to the proof of (37) of \cite{liang2021average}.
	Note that the only difference is for every $k\ge 4$, we have
	\[
	\sum_{i=1}^{k-1}\sqrt{i}\ge \int_{0}^{k-1}\sqrt{x}dx=\frac23(k-1)^{3/2} \ge\frac23\left( \frac {3}{4}k\right) ^{3/2}\ge \frac13 k^{3/2}.
	\]
	
	c) This is Lemma 4.6 of \cite{liang2021average}.
\end{proof}


The following result follows by combining the
conclusions (a) and (b) of
Lemma \ref{lem:basic}, and using Lemma \ref{estimates}(a).

\begin{lemma}\label{lem:sum}
	For every $ u\in{\cal H}  $ and $ k\ge 1 $, we have
	\begin{equation}\label{ineq:sum}
		\sum_{i\in {\cal G}_k}\left( A_{i+1}M_i\|\ty_{i+1}-\tx_i\|^2\right) 
		\le 10\left( d_0^2 +  \frac{\bar m + \bar M}{M_k^{hm}}  D_\Omega^2k + \frac{\bar m}{M_k^{hm}} D_{\cal H} ^2 k^2 \right).
	\end{equation}
\end{lemma}
\begin{proof}
   	Let $ x_* \in X_*$ be such that $ x_*=\argmin\{\|x_0-u\|: u\in X_* \} $ be given and denote $ \|x_0-x_*\| $ by $ d_0 $.
    In view of the definition of $ \eta_k $ in \eqref{eq:eta}, we observe that $ \eta_k(x_*)\ge 0 $ for every $ k\ge 0 $ and $ \eta_0(x_*)=d_0^2/2 $.
	Adding \eqref{ineq:case1} and \eqref{ineq:case2} with $ u=x_* $
	as $ k $ varies in $ {\cal G}_k \cup {\cal B}_k $, and
	using the previous observation and the definition of $ \eta_k $ in \eqref{eq:eta}, we have that for $ k\ge 1 $,
	\[
		\sum_{i\in {\cal G}_k}\left( A_{i+1} M_i \|\ty_{i+1} - \tx_i\|^2\right) 
		\le 10\left( d_0^2 + (\bar m + \bar M)D_\Omega^2\sum_{i=0}^{k-1}\frac{1}{M_i} + \bar mD_{\cal H} ^2A_k \right).
	\]
	Inequality \eqref{ineq:sum} now follows from the above conclusion,
	the definition of $ M_k^{hm} $ in \eqref{eq:L},
	and the second inequality in Lemma \ref{estimates}(a).
\end{proof}

The two following
technical results require Condition A
to hold. 
Recall that sufficient conditions for Condition A to hold have been discussed in the paragraph containing \eqref{eq:alpha}.
Moreover, Subsection \ref{subsec:practice} discusses the likelihood that Condition A holds in the practical setting of AC-FISTA.

Recall from
the discussion on the line above \eqref{eq:alpha} that Condition A always holds with $ k_0=12 $ if $ \alpha $ is chosen so as to satisfy
\eqref{eq:alpha}.
However, our analysis may also hold
for $\alpha$'s
that do not
satisfy the restrictive condition
\eqref{eq:alpha} as long
as the resulting sequence
$\{|{\cal B}_k|\}$ satisfy Condition A
(e.g., see the last paragraph in Subsection \ref{subsec:practice}
which argues that this condition practically holds
for the $(0.5,10^{-6})$-AC-FISTA).


\begin{lemma}\label{lem:Gk}
Assume that Condition A holds and, 
	for every $k\ge 1$, define 
	\[
		\bar {\cal G}_k:=\{i\in{\cal G}_k :  i\ge \lceil k/3 \rceil \}.
	\]
    Then, $ |\bar {\cal G}_k| \ge k/4 $ for every $ k\ge \max\{12,k_0\} $.
\end{lemma}
\begin{proof}
	Using the definitions of $ {\cal B}_k $ in \eqref{def:Gk} and $ \bar {\cal G}_k $ above, we have
 \[
 \bar {\cal G}_k \cup {\cal B}_k \supset \left\{ \left \lceil \frac{k}{3} \right \rceil, \ldots, k-1 \right\},
 \]
 and hence that
 \[
 |\bar {\cal G}_k| +
 | {\cal B}_k| =
 |\bar {\cal G}_k \cup {\cal B}_k | \ge
 k-\left \lceil \frac{k}{3} \right \rceil \ge \frac{2k}3 -1.
 \]
	This observation and Condition A then imply that for every $ k\ge \max\{12,k_0\} $,
	\[
	|\bar {\cal G}_k| \ge \frac{2k}{3} -1 -|{\cal B}_k| \ge  \frac k3 -1  \ge \frac k4,
	\]
 where the last inequality is due to the fact that $k\ge 12$.
\end{proof}



We now present an important inequality of our analysis that connects other key ingredients, i.e., Lemmas \ref{lem:sum}, \ref{lem:Gk}, and \ref{lem:new} below, for proving Theorem \ref{thm:main}.

\begin{lemma}\label{lem:crux}
Under Condition A, we have for every $ k\ge \max\{12,k_0\} $,
\begin{equation}\label{ineq:min}
    \min_{1 \le i \le k} \|v_i\| 
		\le \frac{8}{k^{3/2}} \left( \sum_{ i =\lceil k/3 \rceil }^{k-1}\frac{M_i+ L_i}{\sqrt{A_{i+1}M_i}}\right) \left( \sum_{ i \in {\cal G}_k} A_{i+1}M_i\|\ty_{i+1}-\tx_i\|^2\right)^{1/2}
\end{equation}
where $ L_i := {\cal L}(y^g_{i+1}; \tx_k)$ and ${\cal L}(\cdot;\cdot)$ is as in \eqref{eq:tC}.
\end{lemma}

\begin{proof}
    It follows from the definitions of $ L_k $ and $v_{k+1}$ in this lemma and \eqref{eq:vk}, respectively, and the triangle inequality that
	\[
	\|v_{k+1}\|\le (M_k + L_k)\|y_{k+1}^g - \tx_k\|.
	\]
	Using the above inequality, and the facts that $ \ty_{i+1}=y_{i+1}^g $ for $ i\in {\cal G}_k $ and $\bar {\cal G}_k \subset {\cal G}_k$, we have
	\begin{align}
	\min_{1 \le i \le k} \|v_i\| & \le \min_{ i \in \bar {\cal G}_k} \|v_i\|
	\le \min_{ i \in \bar {\cal G}_k}\left( \frac{M_i+ L_i}{\sqrt{A_{i+1}M_i}}\right) \left( \sqrt{A_{i+1}M_i}\|\ty_{i+1}-\tx_i\|\right) \nn \\
	&\le |\bar {\cal G}_k|^{-3/2} \left( \sum_{ i \in \bar {\cal G}_k}\frac{M_i+ L_i}{\sqrt{A_{i+1}M_i}}\right) \left( \sum_{ i \in \bar {\cal G}_k} A_{i+1}M_i\|\ty_{i+1}-\tx_i\|^2\right)^{1/2} \label{ineq:all1}
	\end{align}
	where the last inequality is due to Lemma 9 of \cite{kong2022accelerated} with $ k=| \bar {\cal G}_k| $, $p=3/2$, and
	\[
	a_i=\frac{M_i+ L_i}{\sqrt{A_{i+1}M_i}}, \quad b_i = \sqrt{A_{i+1}M_i}\|\ty_{i+1}-\tx_i\|.
	\]
    The conclusion of the proposition now follows from \eqref{ineq:all1}, the facts that $ \bar {\cal G}_k \subset \{\lceil k/3 \rceil, \ldots, k-1 \} $ and $ \bar {\cal G}_k \subset {\cal G}_k$, and Lemma \ref{lem:Gk}.
\end{proof}


In view of Lemmas \ref{lem:sum} and \ref{lem:crux},
it is sufficient to develop a bound on the first summation in \eqref{ineq:min} to obtain a bound on $\min_{1 \le i \le k} \|v_i\|$. Hence, we present the following lemma.

\begin{lemma}\label{lem:new}
	For every $ k\ge 12 $, we have
	\begin{equation}\label{ineq:sum1}
		\sum_{ i =\lceil k/3 \rceil}^{k-1}  \frac{M_i+ L_i}{\sqrt{A_{i+1}M_i}}\le 6\sqrt{3}\left( 2M_k + L_k^{avg}\right) 
	\end{equation}
	where $L_i$ is defined in Lemma \ref{lem:crux} and $ L_k^{avg} $ is as in \eqref{eq:L}.
\end{lemma}
\begin{proof}
	In view of the assumption that $k\ge 12$, it is easy to see that $i\ge 4$ for $i\ge \lceil k/3 \rceil$.
	This observation, the fact that $ A_{i+1}\ge A_i $
	and Lemma \ref{estimates}(b) imply that for every $ k\ge 12 $, 
	\[
	\sum_{ i =\lceil k/3 \rceil}^{k-1} \frac{M_i+ L_i}{\sqrt{A_{i+1}M_i}}\le \sum_{ i =\lceil k/3 \rceil}^{k-1}  \frac{M_i+ L_i}{\sqrt{A_{i}M_i}} \le 2\sqrt{3} \sum_{ i =\lceil k/3 \rceil}^{k-1} \frac{M_i+ L_i}{i}.
	\]
	Using Lemma \ref{estimates}(c), we have
	\[
	\sum_{ i =\lceil k/3 \rceil}^{k-1}  \frac{M_i}{i} \le \sum_{ i =\lceil k/3 \rceil}^{k-1}  \frac{kM_k}{i^2} 
	\le k M_k \frac{2k/3}{(k/3)^2} = 6M_k.
	\]
	It is easy to see from the definition of $ L_k^{avg} $ in \eqref{eq:L} that
	\[
	\sum_{ i =\lceil k/3 \rceil}^{k-1}  \frac{L_i}{i}\le \frac{3}{k}\sum_{ i =\lceil k/3 \rceil}^{k-1}  L_i
	\le \frac{3}{k}\sum_{ i =0}^{k-1} L_i = 3 L_k^{avg}.
	\]
	Inequality \eqref{ineq:sum1} immediately follows from the above three inequalities.
\end{proof}


We are now ready to prove Theorem \ref{thm:main}.

\vgap

\noindent
{\bf Proof of Theorem \ref{thm:main}:}
	(a) See Lemma \ref{lem:vk}.
	
	(b) Putting together Lemmas \ref{lem:sum}, \ref{lem:crux}, and \ref{lem:new}, and using the inequality $\sqrt{a+b} \le \sqrt{a} + \sqrt{b}$ for every $a, b\ge 0$, we have
	\begin{align*}
		\min_{1 \le i \le k} \|v_i\| 
		&\le \frac{48\sqrt{30}}{k^{3/2}}  \left( 2M_k + L_k^{avg} \right) \left( d_0 + \frac{\sqrt{\bar m } + \sqrt{\bar M }}{\sqrt{M_k^{hm} }}  D_\Omega \sqrt{k} + \frac{\sqrt{\bar m}}{\sqrt{M_k^{hm} }} D_{\cal H}  k \right).
	\end{align*}
	Statement b) now follows from the above inequality and the definitions of $\theta_k$ and $\tau_k$ in \eqref{def:theta}.

(c) Using the definition of $ M_{k+1} $ in \eqref{eq:M}, and the facts that $ \gamma < 1 $ and $C_k \le \bar M \le M$ for every $ k \ge 0 $, we have for every $k \ge 0$,
	\[
		M_{k+1} = {\cal O}\left( \gamma M + \frac{\bar M}{\alpha}  \right) = {\cal O}\left( \frac{M}{\alpha}  \right),
	\]
 and hence the inequality on $M_k$ in \eqref{ineq:tau} holds.
 Using the definition of $ M_k^{hm} $ in \eqref{eq:L} and Lemma \ref{estimates}(c), we have
	\[
	M_k^{hm} = \frac{k}{\sum_{i=0}^{k-1} \frac{1}{M_i}} \le \frac{k}{\sum_{i=0}^{k-1} \frac{i}{k M_k}} = \frac{2k}{k-1} M_k,
	\]
	and hence the first inequality on $ \theta_k $ in \eqref{ineq:tau} holds in view of the definition of $ \theta_k $. The second inequality on $ \theta_k $ in \eqref{ineq:tau} immediately follows from the definition of $ \theta_k $ in \eqref{def:theta} and the fact that $ M_i\ge \gamma M $ for every $ i\ge 0 $ (see \eqref{eq:M}).
	The bound on $ \tau_k $ in \eqref{ineq:tau} is a direct consequence of the definition of $ \tau_k $ in \eqref{def:theta}, and the facts that $ L_k^{avg} \le \bar L $ and $ M_k\ge \gamma M $.
 Finally, the bound $ {\cal O}(1/(\alpha \gamma)) $ on $ \theta_k $ immediately follows from the bound on $M_k$ and the second inequality on $ \theta_k $ in \eqref{ineq:tau}.
\QEDA

\section{Numerical results}\label{sec:numerics}


This section reports computational results of AC-FISTA and
a corresponding restart variant against
five other state-of-the-art algorithms on three instances of N-SCO problems: support vector machine (Subsection \ref{subsec:SVM}),
quadratic programming (Subsection \ref{subsec:QP}) and matrix completion (Subsection \ref{subsec:MC}).

We start by describing the implementation of AC-FISTA and its restart variant used in our computational benchmark.
Our implementation of AC-FISTA 
sets $M_0=0.01M$, computes $M_{k+1}$ according to \eqref{eq:M} with $(\alpha,\gamma)=(0.5,10^{-6})$,
and chooses $ y_{k+1}=\ty_{k+1} $.
The restart variant of AC-FISTA uses the same parameters as AC-FISTA
but rejects $ y_{k+1} $ whenever $ k\in {\cal G} $ and $ \phi(y_{k+1})\ge \phi(y_k) $ in which case it sets
$ x_k=y_k$ and $ A_k=0$, and repeats the $k$-th iteration.

We compare our methods with five other ACG variants, namely: 
(i) the UPFAG method in \cite{LanUniformly};
(ii) the ADAP-NC-FISTA described in \cite{liang2019fista};
(iii) the theoretical AC-ACG method proposed in \cite{liang2021average} (referred
to as ACT in its Section 5);
and (iv)
restart variants of the methods
in (ii) and (iii) which are
described in the paragraph
below.
For the sake of simplicity, we use the
abbreviations UP, AD, AC and AF to refer to UPFAG, ADAP-NC-FISTA, AC-ACG and AC-FISTA, respectively, both in the discussions and tables below. Moreover, we use AD(R), AC(R) and AF(R) to denote the restart variants of AD, AC and AF, respectively.

This paragraph provides details about the five other ACG variants used in our benchmark.
UP is described in Algorithm 1 of \cite{LanUniformly} and
the code for it was provided by the authors of \cite{LanUniformly}. In particular, we have used their choice of parameters
but have slightly modified the code to accommodate for our termination criterion, i.e., Definition \ref{def:approx pt}.
More specifically, the input parameters $ (\hat \lam_0, \hat \beta_0, \gamma_1,\gamma_2,\gamma_3,\delta,\gamma)$  of UP
were set to
$ (1/L,1/L,0.4,0.4, 1,10^{-3},10^{-10}) $.
AD was implemented by
the authors according to its description in Section 3 in \cite{liang2019fista}.
The input triple $ (M_0, m_0, \theta) $ of AD was 
set to $ (1,1000,1.25) $ in Subsections \ref{subsec:SVM} and \ref{subsec:QP}, and $ (1,1,1.25) $ in Subsection \ref{subsec:MC}.
Method AC is exactly the theoretical AC-ACG method of \cite{liang2021average} with parameter pair $(\alpha,\gamma)$ set to $(0.5,0.01)$.
Moreover, the restart variant AC(R) (resp., AD(R)) uses the same set of parameters as AC (resp., AD) and restarts in the same way as AF(R) does.

All seven methods terminate when a pair $(z,v)$ satisfying a relative termination criterion
\[
v\in \nabla f(z)+\partial h(z), \qquad  \frac{\|v\|}{\|\nabla f(z_{0})\|+1}\leq \hat \rho
\]
is obtained,
where $z_0$ is the initial point, $ \hat \rho=10^{-7} $ in Subsections \ref{subsec:SVM} and \ref{subsec:QP}, and $ \hat \rho=5\times 10^{-4} $ in Subsection \ref{subsec:MC}.
We run all numerical experiments using
MATLAB R2017b on a MacBook Pro with a quad-core Intel Core i7 processor and 16 GB of memory.


\subsection{Support Vector Machine} \label{subsec:SVM}
This subsection discusses the performance of the methods in our computational benchmark for solving a support vector machine (SVM) problem (see (4.1) in \cite{LanUniformly}). For given data points $ \{(u_i,v_i)\}_{i=1}^p $, where $ u_i\in \R^n $ is a feature vector and $ v_i\in \{-1,1\} $ denotes the corresponding label, we formulate the SVM
problem as
\begin{equation}\label{eq:SVM}
	\min \left\lbrace f(z):= \frac1p \sum_{i=1}^{p} \ell(u_i,v_i;z)+\frac{\lam}{2}\|z\|^2: z\in B_r \right\rbrace 
\end{equation}
where $ \ell(u_i,v_i;\cdot)=1-\tanh(v_i\inner{\cdot}{u_i}) $ is a nonconvex sigmoid loss function, $ \lam>0 $ is a regularization parameter and $B_r :=\{z\in \R^n:\|z\|\le r\} $ is a ball with radius $ r>0 $ and centered at the origin.

The SVM problem \eqref{eq:SVM} is an instance of \eqref{eq:PenProb2Intro} where
$ h $ is the indicator function of the ball $B_r$. We set $ \lam = 1/p$, $ r = 50 $ and $ \Omega = B_r $, where the set $ \Omega $ is introduced in (A2). 
It can be shown that $ f $ is differentiable everywhere and satisfies
\[
	m = M = L = \frac1p \sum_{i=1}^p \frac{4\sqrt{3}}{9}\|u_i\|^2 + \lam, \quad \forall i=1,\ldots,p.
\]


We now describe the datasets {\it SVM-1}, {\it SVM-2}, {\it SVM-3} and {\it SVM-4} considered in the numerical experiments. Each dataset contains data points $ \{(u_i,v_i)\}_{i=1}^p $ where $ u_i $ is a sparse vector with density $ d $ and its nonzero entries are drawn from the uniform distribution $ {\cal U}[0,1]$, and $ v_i=\text{sign}(\inner{\bar z}{u_i}) $ for some $ \bar z \in B_r$.
Table~\ref{tab:t1} lists basic statistics of the datasets.

\begin{table}[H]
	\begin{centering}
		\begin{tabular}{|>{\centering}p{2cm}|>{\centering}p{1.5cm}|>{\centering}p{1.5cm}|>{\centering}p{1.7cm}|>{\centering}p{1.5cm}|>{\centering}p{1.5cm}|>{\centering}p{1.5cm}|}
			\hline 
			{Dataset} & {$ n $}& {$ p $} & {Density $ d $} & { $ \lam $} & { $ r $} & { $ M$} \tabularnewline
			\hline 
			{\small{{\it SVM-1}}}  & {\small{}5000} & {\small{}500} & {\small{}5\%} & {\small{}0.002} & {\small{}50} & {\small{}13} \tabularnewline
			\hline
			{\small\text{{\it SVM-2}}} & {\small{}10000} & {\small{}1000} & {\small{}5\%} & {\small{}0.001} & {\small{}50} & {\small{}25} \tabularnewline
			\hline
			{\small\text{{\it SVM-3}}} & {\small{}15000} & {\small{}1000} & {\small{}5\%} & {\small{}0.001} & {\small{}50} & {\small{}38} \tabularnewline
			\hline
			{\small\text{{\it SVM-4}}} & {\small{}20000} & {\small{}500} & {\small{}5\%} & {\small{}0.002} & {\small{}50} & {\small{}50} \tabularnewline
			\hline
		\end{tabular}
		\par\end{centering}
	\caption{SVM datasets}\label{tab:t1}
\end{table}
\vspace{-5mm}

We start all seven methods from the same initial point $ z_0 $ that
is generated randomly and uniformly within the ball $ B_r $. 

Numerical results of the seven methods for solving \eqref{eq:SVM} with datasets {\it SVM-1}, {\it SVM-2}, {\it SVM-3} and {\it SVM-4} are given in Table \ref{tab:t2}.
Specifically,  the second to eighth columns
provide numbers of iterations and running times for the seven methods.
We do not report the best objective function values obtained
by all seven methods,
since they are essentially the same
on each instance.
The bold numbers highlight the method that has the best performance
in an instance of \eqref{eq:SVM}.

\begin{table}[H]
	\begin{centering}
		\begin{tabular}{|>{\centering}p{1.1cm}|>{\centering}p{.6cm}>{\centering}p{.7cm}>{\centering}p{.6cm}>{\centering}p{.8cm}>{\centering}p{.8cm}>{\centering}p{.8cm}>{\centering}p{.8cm}|}
			\hline 
			{\small{}Dataset} & \multicolumn{7}{c|}{\makecell{\small{}Iteration Count / \\ \small{}Running Time (s)}}
			\tabularnewline  
			\cline{2-8}
			& {\small{}UP} & {\small{}AD} & {\small{}AC} & {\small{}AF} & {\small{}AD(R)} & {\small{}AC(R)}  & {\small{}AF(R)} 
			\tabularnewline
			\hline 
			{\footnotesize{\it{SVM-1}}} & {\small{}250\\ 18} & {\small{}2333\\  52} & {\small{}1678\\ 62} & {\small{}604\\ 14} & {\small{}440\\  13} & {\small{}339\\ 15} & {\small{}160\\ \textbf{5}}
			\tabularnewline
			\hline 
			{\footnotesize{\it{SVM-2}}} & {\small{}254\\ 81} & {\small{}3996\\  {396}} & {\small{}4801\\ 772} & {\small{}1352\\ 144} & {\small{}549\\ 67} & {\small{}605\\ 110} & {\small{}230\\ \textbf{29}} 
			\tabularnewline
			\hline 
			{\footnotesize{\it{SVM-3}}} & {\small{}284\\ 137} & {\small{}3499\\ 529} & {\small{}6023\\  1505} & {\small{}1563\\ 248} & {\small{}503\\ 93} & {\small{}695\\ 187} & {\small{}200\\ \textbf{38}} 
			\tabularnewline
			\hline 
			{\footnotesize{\it{SVM-4}}} & {\small{}156\\ 50} & {\small{}1701\\  175} & {\small{}4136\\  661} & {\small{}823\\ 86} & {\small{}377\\ 46} & {\small{}630\\ 113} & {\small{}151\\ \textbf{19}} 
			\tabularnewline
			\hline 
		\end{tabular}
		\par\end{centering}
	\caption{Numerical results for solving \eqref{eq:SVM} with {\it SVM-1, 2, 3,} \& {\it 4}}\label{tab:t2}
\end{table}
\vspace{-3mm}


Recall that we have commented on the practical behavior of the ratios $ \theta_k $, $ \tau_k $ and $ |{\cal B}_k|/k $ in Subsection \ref{subsec:practice}. We now present the statistics of the three ratios of AF and AF(R) for solving the SVM problem \eqref{eq:SVM}.

\begin{table}[H]
	\begin{centering}
		\begin{tabular}{|>{\centering}p{1.2cm}|>{\centering}p{1.2cm}|>{\centering}p{1.2cm}|>{\centering}p{1.2cm}|>{\centering}p{1.2cm}|>{\centering}p{1.2cm}|>{\centering}p{1.2cm}|}
			\hline 
			{\small{}Dataset} & \multicolumn{3}{c|}{AF} & \multicolumn{3}{c|}{AF(R)}
			\tabularnewline  
			\cline{2-7}
			& {\small$ \bar \theta_k $} & {\small $ \bar \tau_k $} & {\small $ |{\cal B}_k|/k $} & {\small$ \bar \theta_k $} & {\small $ \bar \tau_k $} & {\small $ |{\cal B}_k|/k $}
			\tabularnewline
			\hline 
			{\small{{\it SVM-1}}}  & {\small{}1.34} & {\small{}0.55} & {\small{}31\%} & {\small{}2.16} & {\small{}0.55} & {\small{}37\%}  \tabularnewline
			\hline
			{\small{{\it SVM-2}}}  & {\small{}1.16} & {\small{}0.60} & {\small{}32\%} & {\small{}1.24} & {\small{}0.61} & {\small{}35\%}  \tabularnewline
			\hline
			{\small{{\it SVM-3}}}  & {\small{}1.04} & {\small{}0.58} & {\small{}26\%} & {\small{}1.35} & {\small{}0.62} & {\small{}32\%}  \tabularnewline
			\hline
			{\small{{\it SVM-4}}} & {\small{}0.93} & {\small{}0.55} & {\small{}21\%} & {\small{}1.25} & {\small{}0.60} & {\small{}37\%}  \tabularnewline
			\hline
		\end{tabular}
		\par\end{centering}
	\caption{Statistics of $\bar \theta_k$, $\bar \tau_k$ and $ |{\cal B}_k| $}\label{tab:SVM}
\end{table}
\vspace{-5mm}

In Table \ref{tab:SVM}, $ \bar \theta_k $ and $ \bar \tau_k $ are defined as
\[
	\bar \theta_k:= \max\{ \theta_k: k\ge 100 \}, \quad \bar \tau_k:= \max\{ \tau_k: k\ge 100 \}.
\]
The ratio $ |{\cal B}_k|/k $ represents the the percentage of bad iterations at the last iteration of each method.

In summary, computational results demonstrate that:
i) AF(R) is the best method in terms of running time; 
ii) AF(R) (resp., AD(R) an AC(R)) improves the results of AF (resp., AD and AC);
and iii) $ \bar \theta_k $ and $ \bar \tau_k $ are small and $ |{\cal B}_k|/k $ is no more than $ 37\% $.

\subsection{Quadratic Programming} \label{subsec:QP}

In this subsection, we consider solving a class of nonconvex quadratic programming (QP) problems.
More specifically, the QP problem reads as
\begin{equation}\label{testQPprobMat}
	\min\left\{ f(Z):=-\frac{\alpha_1}{2}\|D\mathcal{P}(Z)\|^{2}+\frac{\alpha_2}{2}\|\mathcal{Q}(Z)-b\|^{2}:Z\in O_{n}\right\}
\end{equation}
where $(\alpha_1,\alpha_2)\in \R^2_{++}$, $b\in \R^{ l}$,
$D\in \R^{n\times n}$,
$ O_n:=\{Z\in {\cal S}_+^n:\text{tr}(Z)=1\} $ denotes the spectraplex, and 
$ \mathcal{P}:{\cal S}_+^n\rightarrow \R^n $ and $ \mathcal{Q}:{\cal S}_+^n\rightarrow \R^l $ are linear operators given by
\begin{align*}
	\left[ \mathcal{P}(Z) \right]_i&=\inner{P_i}{Z}_F \quad \forall \, 1\le i\le n, \\
	\left[ \mathcal{Q}(Z) \right]_j&=\inner{Q_j}{Z}_F \quad \forall \, 1\le j\le l, 
\end{align*}
with $ P_i \in {\cal S}^n_+ $ and $ Q_j \in {\cal S}^n_+$.

We now describe the datasets {\it QP-1} and {\it QP-2} considered in the numerical experiments. Each dataset contains $ b $, $ D $, $ P_i $ for $ 1\le i \le n $ and $ Q_j $ for $ 1\le j \le l $.
The entries of $ b $ are sampled from the uniform distribution ${\cal U}[0,1]$.
The diagonal entries of the diagonal matrix $D$ are generated from the discrete uniform distribution ${\cal U}\{1,1000\}$.
Sparse matrices $ P_i $ and $ Q_j $ have the same density (i.e., percentage of nonzeros) $ d $ and their nonzero entries are generated from ${\cal U}[0,1]$.
Table~\ref{tab:t3} lists basic statistics of the datasets.

\begin{table}[H]
	\begin{centering}
		\begin{tabular}{|>{\centering}p{2cm}|>{\centering}p{1.5cm}|>{\centering}p{1.5cm}|>{\centering}p{1.7cm}|}
			\hline 
			{Dataset} & {$ l $}& {$ n $} & {Density $ d $} \tabularnewline
			\hline 
			{\small{{\it QP-1}}}  & {\small{}50} & {\small{}200} & {\small{}2.5\%} \tabularnewline
			\hline
			{\small\text{{\it QP-2}}} & {\small{}50} & {\small{}400} & {\small{}0.5\%} \tabularnewline
			\hline
		\end{tabular}
		\par\end{centering}
	\caption{Quadratic programming datasets}\label{tab:t3}
\end{table}
\vspace{-5mm}

The nonconvex QP problem \eqref{testQPprobMat} is an instance of \eqref{eq:PenProb2Intro} where $ h$ is the indicator function of the spectraplex $ O_n $.
The set $ \Omega $ introduced in (A2) is chosen as $ \Omega=\{Z\in {\cal S}_+^n: \|Z\|_F=1\} $. It is easy to see that $ \Omega \supset {\cal H}  $, which is required in (A2), since $ \Omega \supset O_n  = {\cal H}  $.
For given curvature pairs $(M,m)\in\R^2_{++}$, we $ L $ set to $ \max\{M,m\} $ and choose scalars $\alpha_1$ and $\alpha_2$  so that
$(\lambda_{\max}(\nabla^{2}f),\lambda_{\min}(\nabla^{2}f)) =(M,-m)$, where $\lambda_{\max}(\cdot)$ (resp., $\lambda_{\min}(\cdot)$)
denotes the largest (resp., smallest) eigenvalue function.

We start all seven methods from the same initial point $ Z_0 = I_n/n $ where $ I_n $ is an $ n\times n $ identity matrix, namely $ Z_0 $ is the centroid of $ O_n $.

Numerical results of the seven methods for solving \eqref{testQPprobMat}  with datasets {\it QP-1} and {\it QP-2} are given
in Tables \ref{tab:t4} and \ref{tab:t5}, respectively. 
Each table addresses a collection of instances with the same dataset and $ M=10^6 $.
Specifically, each table contains six
instances of \eqref{testQPprobMat}, their first column
specifies $m$ for
the instances.
The explanation of columns in Tables \ref{tab:t4} and \ref{tab:t5} excluding the first one is the same as that of Table \ref{tab:t2} (see the paragraphs preceding Table \ref{tab:t2}).
We do not report the best objective function values obtained
by all seven methods,
since they are essentially the same
on each instance.
The bold numbers highlight the method that has the best performance
in an instance of \eqref{testQPprobMat}.
\begin{table}[H]
	\begin{centering}
		\begin{tabular}{|>{\centering}p{.4cm}|>{\centering}p{.6cm}>{\centering}p{.6cm}>{\centering}p{.6cm}>{\centering}p{.6cm}>{\centering}p{.8cm}>{\centering}p{.8cm}>{\centering}p{.8cm}|}
			\hline 
			{\small{}$ m $} & \multicolumn{7}{c|}{\makecell{\small{}Iteration Count / \\ \small{}Running Time (s)}}
			\tabularnewline  
			\cline{2-8}
			& {\small{}UP} & {\small{}AD} & {\small{}AC} & {\small{}AF} & {\small{}AD(R)} & {\small{}AC(R)}  & {\small{}AF(R)} 
			\tabularnewline
			\hline 
			{\small{}$ 10^5 $} & {\small{}2633\\ 261} & {\small{}2206\\  {89}} & {\small{}1009\\  {55}} & {\small{}947\\ 30} & {\small{}787\\ 33} & {\small{}966\\ 55} & {\small{}419\\ \textbf{14}} 
			\tabularnewline
			\hline 
			{\small{}$ 10^4 $} & {\small{}7203\\ 705} & {\small{}2591\\  {104}} & {\small{}1820\\  {98}} & {\small{}1744\\ 55} & {\small{}1573\\ 66} & {\small{}1777\\ 99} & {\small{}601\\ \textbf{20}} 
			\tabularnewline
			\hline 
			{\small{}$ 10^3 $} & {\small{}5429\\ 540} & {\small{}2637\\  {109}} & {\small{}1712\\  {92}} & {\small{}2000\\ 63} & {\small{}1552\\ 65} & {\small{}1709\\ 100} & {\small{}773\\ \textbf{26}} 
			\tabularnewline
			\hline 
			{\small{}$ 10^2 $} & {\small{}6891\\ 653} & {\small{}2639\\  {116}} & {\small{}1610\\  {95}} & {\small{}1687\\ 52} & {\small{}1666\\ 69} & {\small{}1600\\ 96} & {\small{}736\\ \textbf{25}} 
			\tabularnewline
			\hline 
			{\small{}$ 10 $} & {\small{}6479\\ 613} & {\small{}2640\\  {116}} & {\small{}1599\\  {95}} & {\small{}1804\\ 56} & {\small{}1675\\ 69} & {\small{}1593\\ 96} & {\small{}785\\ \textbf{26}} 
			\tabularnewline
			\hline 
		\end{tabular}
		\par\end{centering}
	\caption{Numerical results for solving \eqref{testQPprobMat} with {\it QP-1}}\label{tab:t4}
\end{table}
\vspace{-5mm}

\begin{table}[H]
	\begin{centering}
		\begin{tabular}{|>{\centering}p{.4cm}|>{\centering}p{.6cm}>{\centering}p{.6cm}>{\centering}p{.6cm}>{\centering}p{.6cm}>{\centering}p{.8cm}>{\centering}p{.8cm}>{\centering}p{.8cm}|}
			\hline 
			{\small{}$ m $} & \multicolumn{7}{c|}{\makecell{\small{}Iteration Count / \\ \small{}Running Time (s)}}
			\tabularnewline  
			\cline{2-8}
			& {\small{}UP} & {\small{}AD} & {\small{}AC} & {\small{}AF} & {\small{}AD(R)} & {\small{}AC(R)}  & {\small{}AF(R)} 
			\tabularnewline
			\hline 
			{\small{}$ 10^5 $} & {\small{}56\\ 13} & {\small{}530\\  {56}} & {\small{}403\\  {58}} & {\small{}140\\ \textbf{12}} & {\small{}292\\ 30} & {\small{}414\\ 60} & {\small{}140\\ \textbf{12}}
			\tabularnewline
			\hline 
			{\small{}$ 10^4 $} & {\small{}105\\ 26} & {\small{}868\\  {93}} & {\small{}599\\ 85} & {\small{}195\\ \textbf{17}} & {\small{}364\\ 38} & {\small{}599\\ 86} & {\small{}182\\ \textbf{17}}
			\tabularnewline
			\hline 
			{\small{}$ 10^3 $} & {\small{}115\\ 29} & {\small{}900\\  {103}} & {\small{}564\\ 81} & {\small{}187\\ 16} & {\small{}384\\ 40} & {\small{}557\\ 80} & {\small{}80\\ \textbf{15}}
			\tabularnewline
			\hline 
			{\small{}$ 10^2 $} & {\small{}119\\ 32} & {\small{}904\\  {103}} & {\small{}559\\ 80} & {\small{}216\\ 19} & {\small{}385\\ 40} & {\small{}554\\ 82} & {\small{}179\\ \textbf{16}}
			\tabularnewline
			\hline 
			{\small{}$ 10 $} & {\small{}113 \\  {31}} & {\small{}904 \\  {104}} & {\small{}561\\ 86} & {\small{}221\\ 19} & {\small{}385\\ 40} & {\small{}554\\ 84} & {\small{}177\\ \textbf{16}} 
			\tabularnewline
			\hline 
		\end{tabular}
		\par\end{centering}
	\caption{Numerical results for solving \eqref{testQPprobMat} with {\it QP-2}}\label{tab:t5}
\end{table}
\vspace{-5mm}

We now present the statistics of $ \bar \theta_k $, $ \bar \tau_k $ and $ |{\cal B}_k|/k $ of AF and AF(R) for solving the nonconvex QP problem \eqref{testQPprobMat}.

\begin{table}[H]
	\begin{centering}
		\begin{tabular}{|>{\centering}p{0.4cm}|>{\centering}p{1.2cm}|>{\centering}p{1.2cm}|>{\centering}p{1.2cm}|>{\centering}p{1.2cm}|>{\centering}p{1.2cm}|>{\centering}p{1.2cm}|}
			\hline 
			{\small{}$ m $} & \multicolumn{3}{c|}{AF} & \multicolumn{3}{c|}{AF(R)}
			\tabularnewline  
			\cline{2-7}
			& {\small$ \bar \theta_k $} & {\small $ \bar \tau_k $} & {\small $ |{\cal B}_k|/k $} & {\small$ \bar \theta_k $} & {\small $ \bar \tau_k $} & {\small $ |{\cal B}_k|/k $}
			\tabularnewline
			\hline 
			{\small{} $ 10^5 $}  & {\small{}0.92} & {\small{}1.22} & {\small{}13\%} & {\small{}1.04} & {\small{}1.24} & {\small{}15\%}  \tabularnewline
			\hline
			{\small{} $ 10^4 $}  & {\small{}1.07} & {\small{}1.05} & {\small{}7\%} & {\small{}1.07} & {\small{}1.05} & {\small{}13\%}  \tabularnewline
			\hline
			{\small{} $ 10^3 $}  & {\small{}0.99} & {\small{}1.14} & {\small{}5\%} & {\small{}0.99} & {\small{}1.14} & {\small{}13\%}  \tabularnewline
			\hline
			{\small{} $ 10^2 $}  & {\small{}1.02} & {\small{}1.07} & {\small{}5\%} & {\small{}1.02} & {\small{}1.07} & {\small{}18\%}  \tabularnewline
			\hline
			{\small{} $ 10 $}  & {\small{}1.00} & {\small{}1.10} & {\small{}5\%} & {\small{}1.00} & {\small{}1.10} & {\small{}10\%}  \tabularnewline
			\hline
		\end{tabular}
		\par\end{centering}
	\caption{Statistics of $\bar \theta_k$, $\bar \tau_k$ and $ |{\cal B}_k| $ for {\it QP-1} }\label{tab:QP1}
\end{table}
\vspace{-5mm}

\begin{table}[H]
	\begin{centering}
		\begin{tabular}{|>{\centering}p{0.4cm}|>{\centering}p{1.2cm}|>{\centering}p{1.2cm}|>{\centering}p{1.2cm}|>{\centering}p{1.2cm}|>{\centering}p{1.2cm}|>{\centering}p{1.2cm}|}
			\hline 
			{\small{}$ m $} & \multicolumn{3}{c|}{AF} & \multicolumn{3}{c|}{AF(R)}
			\tabularnewline  
			\cline{2-7}
			& {\small$ \bar \theta_k $} & {\small $ \bar \tau_k $} & {\small $ |{\cal B}_k|/k $} & {\small$ \bar \theta_k $} & {\small $ \bar \tau_k $} & {\small $ |{\cal B}_k|/k $}
			\tabularnewline
			\hline 
			{\small{} $ 10^5 $}  & {\small{}0.60} & {\small{}3.08} & {\small{}13\%} & {\small{}0.60} & {\small{}3.08} & {\small{}13\%}  \tabularnewline
			\hline
			{\small{} $ 10^4 $}  & {\small{}0.68} & {\small{}2.29} & {\small{}18\%} & {\small{}0.72} & {\small{}2.16} & {\small{}15\%}  \tabularnewline
			\hline
			{\small{} $ 10^3 $}  & {\small{}0.69} & {\small{}2.38} & {\small{}16\%} & {\small{}0.74} & {\small{}2.14} & {\small{}15\%}  \tabularnewline
			\hline
			{\small{} $ 10^2 $}  & {\small{}0.69} & {\small{}2.40} & {\small{}14\%} & {\small{}0.73} & {\small{}2.17} & {\small{}15\%}  \tabularnewline
			\hline
			{\small{} $ 10 $}  & {\small{}0.69} & {\small{}2.40} & {\small{}14\%} & {\small{}0.73} & {\small{}2.17} & {\small{}15\%}  \tabularnewline
			\hline
		\end{tabular}
		\par\end{centering}
	\caption{Statistics of $\bar \theta_k$, $\bar \tau_k$ and $ |{\cal B}_k| $ for {\it QP-2} }\label{tab:QP2}
\end{table}
\vspace{-5mm}

In summary, computational results demonstrate that: 
i) AF(R) is the best method in terms of running time; 
ii) AF(R) (resp., AD(R)) improves the results of AF (resp. AD), while AC(R)  has similar performance as AC;
and iii) $ \bar \theta_k $ and $ \bar \tau_k $ are small and $ |{\cal B}_k|/k $ is no more than $ 18\% $.

\subsection{Matrix Completion} \label{subsec:MC}

This subsection considers a constrained version of the nonconvex low-rank matrix completion (NLRMC) problem.

We start by giving a few definitions.
Given parameters $\beta>0$ and $\tau>0$,
let $p:\R\to\R_+ $ denote the log-sum penalty defined as
$
p(t)= p_{\beta,\tau}(t) := \beta\log( 1+|t|/\tau)
$.
Let ${\cal Q}$ denote a subset of $\{1,\ldots,l\} \times \{1,\ldots,n\}$.
Let $ \Pi_{\cal Q}:\R^{l\times n}\to \R^{l\times n} $ denote a linear operator such that, for given $A\in \R^{l\times n}$, $\Pi_{\cal Q}(A)_{ij}=A_{ij}$ if $(i,j)\in {\cal Q} $, and $\Pi_{\cal Q}(A)_{ij}=0$ otherwise.


Given radius $R>0$, penalty parameter $ \mu>0 $, and an incomplete observation matrix $O \in \R^{\cal Q}$,
the constrained version of the NLRMC problem considered in this subsection is
\begin{equation}\label{eq:MC}
	\min \left\lbrace  \frac12 \|\Pi_{\cal Q}(Z-O)\|_F^2+
	\mu\sum_{i=1}^{r}p(\sigma_i(Z)) : Z \in {\cal B}_R \right\rbrace
\end{equation}
where $ r = \min \{l,n\}$, $\sigma_i(Z)$ is the $ i $-th singular value of $ Z $ and ${\cal B}_R = \{ Z \in \R^{l \times n} : \|Z\|_F \le R\}$.

It is discussed in \cite{liang2021average} that \eqref{eq:MC} is an instance of the N-SCO problem \eqref{eq:PenProb2Intro} and can be rewritten as $ \min \{ f(Z)+h(Z): Z\in\R^{l\times n} \} $
where
\begin{align*}
	f(Z)=\frac12 \|\Pi_{\cal Q}(Z-O)\|_F^2+ \mu\sum_{i=1}^{r}[p(\sigma_i(Z))-p_0\sigma_i(Z)], \\
	h(Z)=\mu p_0\|Z\|_* + I_{{\cal B}_R}(Z), \quad p_0=p'(0)=\frac{\beta}{\tau}
\end{align*}
and $ \|\cdot\|_*$ denotes the nuclear norm, i.e.,
$ \|\cdot\|_* := \sum_{i=1}^r \sigma_i(\cdot)$.
It follows from (48) of \cite{liang2021average} that the triple $ (m,M,L) $ satisfying \eqref{ineq:curv} is
\begin{equation}\label{eq:M-MC}
	(m, M, L)=(2\mu\kappa, 1, \max\{1,2\mu\kappa\})
\end{equation}
where $ \kappa=\beta/\tau^2 $.


We now describe the datasets {\it MovieLens 100K}\footnote{http://grouplens.org/datasets/movielens/}
and {\it FilmTrust}\footnote{http://guoguibing.github.io/librec/datasets.html\#filmtrust} considered in the numerical experiments. Each dataset contains an observed index set ${\cal Q}$ and an incomplete observed matrix $ O $ with rows, columns and nonzero entries representing users, items and ratings, respectively, from some collaborative filtering systems.
Table~\ref{tab:t7} lists basic statistics of the datasets.

\begin{table}[H]
	\begin{centering}
		\begin{tabular}{|>{\centering}p{3cm}|>{\centering}p{1.5cm}|>{\centering}p{1.6cm}|>{\centering}p{1.5cm}|>{\centering}p{1.5cm}|>{\centering}p{1.5cm}|}
			\hline 
			{Dataset} & {Users ($ l $)}& {Items (n)} &{Ratings} & {Density} & {Scale} \tabularnewline
			\hline 
			{\small{{\it MovieLens 100K}}}  & {\small{}943} & {\small{}1682} & {\small{}100000} & {\small{}6.30\%} & {\small{}[1,5]} \tabularnewline
			\hline
			{\small\text{{\it FilmTrust}}} & {\small{}1508} & {\small{}2071} & {\small{}35497} & {\small{}1.14\%} & {\small{}[0.5,4.0]} \tabularnewline
			\hline
		\end{tabular}
		\par\end{centering}
	\caption{Matrix completion datasets}\label{tab:t7}
\end{table}
\vspace{-5mm}

The radius $ R $ is chosen as the Frobenius norm of the matrix of size $ l \times n $
containing the same entries as $ O $ in $ {\cal Q} $ and entries outside of $ {\cal Q} $ being maximum of the scale (i.e., $ 5 $ (resp., $ 4 $) in the case of {\it MovieLens 100K} (resp., {\it FilmTrust})).
The set $ \Omega $ introduced in (A2) is set to be $ {\cal B}_R $ with $ R $ as the aforementioned radius.
It is easy to see that $ \Omega \supset {\cal H}  $, which is required in (A2), since $ \Omega={\cal B}_R = {\cal H}  $.

We start all seven methods from the same initial point $ Z_0 $ that is sampled from the standard Gaussian distribution and is within ${\cal B}_R$.

Numerical results of the seven methods for solving \eqref{eq:MC} with datasets {\it MovieLens 100K} and {\it FilmTrust} are given in Tables \ref{tab:t8} and \ref{tab:t9}, respectively. 
Each table addresses a collection of instances with the same dataset.
The first columns in Tables \ref{tab:t8} and \ref{tab:t9} present the values of $ m $ of the four instances computed according to \eqref{eq:M-MC} with four different triples $ (\mu,\beta,\tau) $.
In addition to the numbers of iterations and running times of all seven methods, the second to eighth columns of Tables \ref{tab:t8} and \ref{tab:t9} also provide the function values of \eqref{eq:MC} at the last iteration.
The bold numbers highlight the method that has the best performance (smallest function value or least running time) in an instance of \eqref{eq:MC}.



\begin{table}[H]
	\begin{centering}
		\begin{tabular}{|>{\centering}p{.4cm}|>{\centering}p{.6cm}>{\centering}p{.6cm}>{\centering}p{.8cm}>{\centering}p{.6cm}>{\centering}p{.8cm}>{\centering}p{.8cm}>{\centering}p{.8cm}|}
			\hline 
			{\small{}$m$} & \multicolumn{7}{c|}{\makecell{\small{}Function Value / \\ \small{}Iteration Count / \\ \small{}Running Time (s)}}
			\tabularnewline  
			\cline{2-8}
			& {\small{}UP} & {\small{}AD} & {\small{}AC} & {\small{}AF} & {\small{}AD(R)} & {\small{}AC(R)}  & {\small{}AF(R)}
			\tabularnewline
			\hline 
			{\small{}$ 4.4 $} & {\small{}2605\\ 521\\ 1545} & {\small{}2625\\  {1674}\\ 1946} & {\small{}2296\\  {1046}\\ 1242} & {\small{}\textbf{1836}\\ 375\\ 287} & {\small{}2625\\ 1674\\ 1946} & {\small{}2304\\ 904\\ 1087} & {\small{}1912\\ 305\\ \textbf{245}}
			\tabularnewline \hline
			{\small{}$ 8.9 $} & {\small{}4261\\ 576\\ 1621} & {\small{}4203\\ {1794}\\ 1930} & {\small{}3896\\  {4773}\\ 6519} & {\small{}\textbf{3617}\\ 291\\ 233} & {\small{} 4203\\ 1794\\ 1930} & {\small{}3914\\ 4511\\ 6245} & {\small{}{3797}\\ 241\\ \textbf{208}}
			\tabularnewline
			\hline
			{\small{}$ 20 $} & {\small{}4637\\ 676\\ 1914} & {\small{}4582\\  {2209}\\ 2364} & {\small{}4313\\  14892\\ 19948}  & {\small{}\textbf{4098}\\ 260\\ \textbf{212}} & {\small{} 4582\\ 2209\\ 2364} & {\small{}4312\\ 15708\\ 21666} & {\small{}4164\\ 304\\ {267}}
			\tabularnewline \hline
			{\small{}$ 30 $} & {\small{}6753\\ 606\\ 1628} & {\small{}6293\\ {1963}\\ 2104} & {\small{}6005\\  30815\\ 43172} & {\small{}\textbf{5333}\\  505\\ 417} & {\small{} 6293\\ 1963\\ 2104} & {\small{}5952\\ 27986\\ 38644} & {\small{}{5524}\\  413\\ \textbf{349}}
			\tabularnewline
			\hline
		\end{tabular}
		\par\end{centering}
	\caption{Numerical results for solving \eqref{eq:MC} with {\it MovieLens 100K}}\label{tab:t8}
\end{table}
\vspace{-3mm}

\begin{table}[H]
	\begin{centering}
		\begin{tabular}{|>{\centering}p{.4cm}|>{\centering}p{.6cm}>{\centering}p{.6cm}>{\centering}p{.8cm}>{\centering}p{.6cm}>{\centering}p{.8cm}>{\centering}p{.8cm}>{\centering}p{.8cm}|}
			\hline 
			{\small{}$m$} & \multicolumn{7}{c|}{\makecell{\small{}Function Value / \\ \small{}Iteration Count / \\ \small{}Running Time (s)}}
			\tabularnewline  
			\cline{2-8}
			& {\small{}UP} & {\small{}AD} & {\small{}AC} & {\small{}AF} & {\small{}AD(R)} & {\small{}AC(R)}  & {\small{}AF(R)}
			\tabularnewline
			\hline 
			{\small{}$ 4.4 $} & {\small{}1050\\ 584\\ 6460} & {\small{}1069\\  {2025}\\ 9063} & {\small{}981\\  {942}\\ 4072} & {\small{}{849}\\ 347\\ \textbf{991}} & {\small{}1069\\ 2025\\ 9063} & {\small{}988\\ 1053\\ 4546} & {\small{}\textbf{804}\\ 586\\ 1753}
			\tabularnewline \hline
			{\small{}$ 8.9 $} & {\small{}1814\\ 634\\ 7130} & {\small{}1854\\ {2410}\\ 11171} & {\small{}1759\\  {4312}\\ 22187} & {\small{}{1538}\\ 469\\ \textbf{1334}} & {\small{} 1854\\ 2410\\ 11171} & {\small{}1738\\ 5461\\ 29569} & {\small{}\textbf{1516}\\ 753\\ 2198}
			\tabularnewline
			\hline
			{\small{}$ 20 $} & {\small{}2120\\ 630\\ 7214} & {\small{}2064\\  {2665}\\ 12701} & {\small{}1988\\  13957\\ 73023}  & {\small{}\textbf{1739}\\ 676\\ 1959} & {\small{} 2064\\ 2665\\ 12701} & {\small{}{1993}\\  14379\\ 77128} & {\small{}1777\\ 528\\ \textbf{1617}}
			\tabularnewline \hline
			{\small{}$ 30 $} & {\small{}2980\\ 559\\ 6244} & {\small{}2917\\ {2365}\\ 11205} & {\small{}2855\\  19419\\ 100580} & {\small{}\textbf{2593}\\  533\\ \textbf{1582}} & {\small{} 2917\\ 2365\\ 11205} & {\small{}2853\\ 18515\\ 96675} & {\small{}\textbf{2593}\\  533\\ \textbf{1582}}
			\tabularnewline
			\hline
		\end{tabular}
		\par\end{centering}
	\caption{Numerical results for solving \eqref{eq:MC} with {\it FilmTrust}}\label{tab:t9}
\end{table}
\vspace{-3mm}

We now present the statistics of $ \bar \theta_k $, $ \bar \tau_k $ and $ |{\cal B}_k|/k $ of AF and AF(R) for solving the MC problem \eqref{eq:MC}. 

\begin{table}[H]
	\begin{centering}
		\begin{tabular}{|>{\centering}p{0.4cm}|>{\centering}p{1.2cm}|>{\centering}p{1.2cm}|>{\centering}p{1.2cm}|>{\centering}p{1.2cm}|>{\centering}p{1.2cm}|>{\centering}p{1.2cm}|}
			\hline 
			{\small{}$ m $} & \multicolumn{3}{c|}{AF} & \multicolumn{3}{c|}{AF(R)}
			\tabularnewline  
			\cline{2-7}
			& {\small$ \bar \theta_k $} & {\small $ \bar \tau_k $} & {\small $ |{\cal B}_k|/k $} & {\small$ \bar \theta_k $} & {\small $ \bar \tau_k $} & {\small $ |{\cal B}_k|/k $}
			\tabularnewline
			\hline 
			{\small{} $ 4.4 $}  & {\small{}1.07} & {\small{}1.23} & {\small{}6\%} & {\small{}1.12} & {\small{}1.20} & {\small{}4\%}  \tabularnewline
			\hline
			{\small{} $ 8.9 $}  & {\small{}1.04} & {\small{}1.53} & {\small{}8\%} & {\small{}1.02} & {\small{}1.48} & {\small{}10\%}  \tabularnewline
			\hline
			{\small{} $ 20 $}  & {\small{}0.97} & {\small{}2.16} & {\small{}9\%} & {\small{}1.00} & {\small{}1.88} & {\small{}13\%}  \tabularnewline
			\hline
			{\small{} $ 30 $}  & {\small{}1.02} & {\small{}2.49} & {\small{}7\%} & {\small{}1.02} & {\small{}2.40} & {\small{}11\%}  \tabularnewline
			\hline
		\end{tabular}
		\par\end{centering}
	\caption{Statistics of $\bar \theta_k$, $\bar \tau_k$ and $ |{\cal B}_k| $ for {\it MovieLens 100K} }\label{tab:MC1}
\end{table}
\vspace{-5mm}

\begin{table}[H]
	\begin{centering}
		\begin{tabular}{|>{\centering}p{0.4cm}|>{\centering}p{1.2cm}|>{\centering}p{1.2cm}|>{\centering}p{1.2cm}|>{\centering}p{1.2cm}|>{\centering}p{1.2cm}|>{\centering}p{1.2cm}|}
			\hline 
			{\small{}$ m $} & \multicolumn{3}{c|}{AF} & \multicolumn{3}{c|}{AF(R)}
			\tabularnewline  
			\cline{2-7}
			& {\small$ \bar \theta_k $} & {\small $ \bar \tau_k $} & {\small $ |{\cal B}_k|/k $} & {\small$ \bar \theta_k $} & {\small $ \bar \tau_k $} & {\small $ |{\cal B}_k|/k $}
			\tabularnewline
			\hline 
			{\small{} $ 4.4 $}  & {\small{}1.09} & {\small{}1.25} & {\small{}10\%} & {\small{}1.11} & {\small{}1.21} & {\small{}9\%}  \tabularnewline
			\hline
			{\small{} $ 8.9 $}  & {\small{}1.02} & {\small{}1.55} & {\small{}6\%} & {\small{}0.99} & {\small{}1.61} & {\small{}6\%}  \tabularnewline
			\hline
			{\small{} $ 20 $}  & {\small{}1.04} & {\small{}2.07} & {\small{}8\%} & {\small{}1.06} & {\small{}2.07} & {\small{}9\%}  \tabularnewline
			\hline
			{\small{} $ 30 $}  & {\small{}1.04} & {\small{}2.59} & {\small{}11\%} & {\small{}1.04} & {\small{}2.59} & {\small{}11\%}  \tabularnewline
			\hline
		\end{tabular}
		\par\end{centering}
	\caption{Statistics of $\bar \theta_k$, $\bar \tau_k$ and $ |{\cal B}_k| $ for {\it FilmTrust} }\label{tab:MC2}
\end{table}
\vspace{-5mm}


In summary, computational results demonstrate that: 
i) AF and AF(R) are the best two methods; 
ii) AD(R) does not restart and has the same performance as AD;
and iii) $ \bar \theta_k $ and $ \bar \tau_k $ are small and $ |{\cal B}_k|/k $ is no more than $ 13\% $.


\section{Concluding remarks}\label{sec:conclusion}
This paper studies the AC-FISTA method, which is
a FISTA-type ACG variant of the AC-ACG method
proposed in \cite{liang2021average},
for solving the N-SCO problem \eqref{eq:PenProb2Intro}.
At the $k$-th iteration, both methods
compute $y(\tx_k;M_k)$ defined in \eqref{eq:update} as a potential
candidate for the next iterate where
$M_k$ is an estimation of the local upper curvature of \eqref{eq:PenProb2Intro}
at $\tx_k$
obtained according to \eqref{eq:M}, and chooses
as the next iterate either
this point if it satisfies \eqref{ineq:descent} or the convex combination in \eqref{eq:ty} otherwise.
However, in contrast to AC-ACG, AC-FISTA computes $M_k$ according
to \eqref{eq:M} using the average of the observed upper curvatures
$C_k$'s defined in \eqref{eq:C} instead of the larger
upper-Lipschitz curvatures $\tilde C_k$'s defined in
the line above \eqref{eq:tC}. 
In addition, AC-FISTA performs only one composite resolvent evaluation during the  good iterations, and two composite resolvent evaluations in the bad ones, but has been
observed to perform an average of
about one composite resolvent evaluation per iteration in practice.
These two features together lead to a practical AC-FISTA variant that
substantially outperforms
previous ACG variants as well as the theoretical and practical
AC-ACG variants,
both in terms of running time and  solution quality.

We end this paper by discussing some possible extensions. 
First, even though we have not studied the convergence rate of
the practical AC-ACG variant of \cite{liang2021average},
we believe that such analysis will follow by using similar arguments
as the ones used in this paper to analyze AC-FISTA.
Second, numerical results show that the restart variant of AC-FISTA greatly improves the empirical performance of its original variant
but its convergence rate analysis has not been established
anywhere in the literature and is an interesting research
direction to pursue.

\section*{Data availability statements}

The datasets generated during and/or analyzed during the current study are available from the corresponding author on reasonable request.

\bibliographystyle{plain}
\bibliography{Proxacc_ref}

\def\cprime{$'$}
\begin{thebibliography}{10}

\bibitem{carmon2018accelerated}
Y.~Carmon, J.~C. Duchi, O.~Hinder, and A.~Sidford.
\newblock Accelerated methods for nonconvex optimization.
\newblock {\em SIAM Journal on Optimization}, 28(2):1751--1772, 2018.

\bibitem{Paquette2017}
D.~Drusvyatskiy and C.~Paquette.
\newblock Efficiency of minimizing compositions of convex functions and smooth
  maps.
\newblock {\em Mathematical Programming}, pages 1--56, 2018.

\bibitem{nonconv_lan16}
S.~Ghadimi and G.~Lan.
\newblock Accelerated gradient methods for nonconvex nonlinear and stochastic
  programming.
\newblock {\em Math. Programming}, 156:59--99, 2016.

\bibitem{LanUniformly}
S.~Ghadimi, G.~Lan, and H.~Zhang.
\newblock Generalized uniformly optimal methods for nonlinear programming.
\newblock {\em Journal of Scientific Computing}, 79(3):1854--1881, 2019.

\bibitem{gillis2014and}
N.~Gillis.
\newblock The why and how of nonnegative matrix factorization.
\newblock {\em Regularization, Optimization, Kernels, and Support Vector
  Machines}, 12(257):257--291, 2014.

\bibitem{gu2014sparse}
Q.~Gu, Z.~Wang, and H.~Liu.
\newblock Sparse pca with oracle property.
\newblock In {\em Advances in neural information processing systems}, pages
  1529--1537, 2014.

\bibitem{KongMeloMonteiro}
W.~Kong, J.~G. Melo, and R.~D.~C. Monteiro.
\newblock Complexity of a quadratic penalty accelerated inexact proximal point
  method for solving linearly constrained nonconvex composite programs.
\newblock {\em SIAM Journal on Optimization}, 29(4):2566--2593, 2019.

\bibitem{kong2022accelerated}
W.~Kong and R.~D.~C. Monteiro.
\newblock Accelerated inexact composite gradient methods for nonconvex spectral
  optimization problems.
\newblock {\em Computational Optimization and Applications}, pages 1--43, 2022.

\bibitem{lee1999learning}
D.~D. Lee and H.~S. Seung.
\newblock Learning the parts of objects by non-negative matrix factorization.
\newblock {\em Nature}, 401(6755):788, 1999.

\bibitem{Li_Lin2015}
H.~Li and Z.~Lin.
\newblock Accelerated proximal gradient methods for nonconvex programming.
\newblock {\em Adv. Neural Inf. Process. Syst.}, 28:379--387, 2015.

\bibitem{Lin_Zhou_Liang_Varshney}
Q.~Li, Y.~Zhou, Y.~Liang, and P.~K. Varshney.
\newblock Convergence analysis of proximal gradient with momentum for nonconvex
  optimization.
\newblock In {\em Proceedings of the 34th International Conference on Machine
  Learning-Volume 70}, pages 2111--2119, 2017.

\bibitem{jliang2018double}
J.~Liang and R.~D.~C. Monteiro.
\newblock A doubly accelerated inexact proximal point method for nonconvex
  composite optimization problems.
\newblock {\em Available on arXiv:1811.11378}, 2018.

\bibitem{liang2021average}
J.~Liang and R.~D.~C. Monteiro.
\newblock An average curvature accelerated composite gradient method for
  nonconvex smooth composite optimization problems.
\newblock {\em SIAM Journal on Optimization}, 31(1):217--243, 2021.

\bibitem{liang2019fista}
J.~Liang, R.~D.~C. Monteiro, and C.-K. Sim.
\newblock A {F}{I}{S}{T}{A}-type accelerated gradient algorithm for solving
  smooth nonconvex composite optimization problems.
\newblock {\em Computational Optimization and Applications}, pages 1--31, 2021.

\bibitem{paquette2018catalyst}
C.~Paquette, H.~Lin, D.~Drusvyatskiy, J.~Mairal, and Z.~Harchaoui.
\newblock Catalyst for gradient-based nonconvex optimization.
\newblock In {\em International Conference on Artificial Intelligence and
  Statistics}, pages 613--622. PMLR, 2018.

\bibitem{yao2017efficient}
Q.~Yao and J.~T. Kwok.
\newblock Efficient learning with a family of nonconvex regularizers by
  redistributing nonconvexity.
\newblock {\em Journal of Machine Learning Research}, 18:179--1, 2017.

\bibitem{Yao_et.al.}
Q.~Yao, J.~T. Kwok, F.~Gao, W.~Chen, and T.-Y. Liu.
\newblock Efficient inexact proximal gradient algorithm for nonconvex problems.
\newblock In {\em Proceedings of the Twenty-Sixth International Joint
  Conference on Artificial Intelligence}, pages 3308--3314. IJCAI, 2017.

\end{thebibliography}


\end{document}